\documentclass[11pt,reqno]{amsart}
\usepackage{amsmath, amsthm, amssymb, stmaryrd}
\usepackage{enumerate}
\usepackage{url}

\newcommand{\caseA}{{\rm (A)}}
\newcommand{\caseB}{{\rm (B)}}
\newcommand{\xzero}{x_0}
\newcommand{\justx}{x}
\newcommand{\red}[1]{#1}

\newcommand{\SL}[2]{\mathrm{SL}_{#1}(#2)}

\newcommand{\lat}[1]{#1\Z^{n+1}}

\let\OLDthebibliography\thebibliography
\renewcommand\thebibliography[1]{
  \OLDthebibliography{#1}
  \setlength{\parskip}{0pt}
  \setlength{\itemsep}{0pt plus 0.3ex}
}

\usepackage{mathtools}

\usepackage{multirow, array}
\usepackage{placeins}
\usepackage{xcolor}

\usepackage{caption}

\advance \topmargin by -\headheight
\advance \topmargin by -\headsep

\evensidemargin \oddsidemargin
\newtheorem{thm}{Theorem}[section]
\newtheorem{lemma}[thm]{Lemma}
\newtheorem{prop}[thm]{Proposition}
\newtheorem{cor}[thm]{Corollary}

\theoremstyle{plain}
\newtheorem{prob}{Problem}

\theoremstyle{definition}
\newtheorem{defn}[thm]{Definition}

\theoremstyle{remark}
\newtheorem{remark}[thm]{Remark}

\numberwithin{equation}{section}

\newcommand{\h}{h}
\newcommand{\HH}{h}
\newcommand{\we}{\wedge}

\newcommand{\vv}[1]{\mathbf{#1}}
\newcommand{\rr}{\mathbf{r}}

\newcommand*\wrapletters[1]{\wr@pletters#1\@nil}
\def\wr@pletters#1#2\@nil{#1\allowbreak\if&#2&\else\wr@pletters#2\@nil\fi}

\def \sharp {\#} 
\def \Span{\operatorname{Span}} 

\def\eps{\varepsilon} \def \epsilon {{\varepsilon}}

\newcommand{\C}{\mathbb{C}}         
\newcommand{\N}{\mathbb{N}}         
\newcommand{\I}{\mathbb{I}}
\newcommand{\R}{\mathbb{R}}        
\newcommand{\Z}{\mathbb{Z}}         

\def \bh {\mathbf h}

\def \cB {\mathcal B}

\def \cH {\mathcal H}
\def \cI {\mathcal I}
\def \cJ {\mathcal J}
\def \cK {\mathcal K}

\def \cS {\mathcal S}

\def \cW {\mathcal W}

\def \dim {\mathrm{dim}}

\def \det {\mathrm{det}}

\def \deg {\mathrm{deg}}

\def \diag {\mathrm{diag}}

\def \supp {\operatorname{supp}}
\def \Bad {{\mathbf{Bad}}}
\def \R {{\mathbb{R}}}
\def \Z {{\mathbb{Z}}}

\def \I {{\mathrm{I}}}

\newcommand{\Par}{\mathrm{\mathbf{Par}}}

\newcommand{\Iq}{I}
\newcommand{\Ip}{I'}
\newcommand{\Jp}{J}
\newcommand{\Jq}{J}
\newcommand{\Jqo}{J}

\renewcommand{\ell}{l}

\newcommand{\m}{1}
\newcommand{\mo}{}

\newcommand{\etaq}{q/4}
\newcommand{\etaqt}{q/8}
\newcommand{\etaqf}{q/16}
\newcommand{\etainv}{4}
\newcommand{\etainvt}{8}

\begin{document}

\title{Winning property of badly approximable points on curves}

\author[Victor Beresnevich]{Victor Beresnevich}
\address{Victor Beresnevich, Department of Mathematics, University of York, Heslington, York, YO10 5DD, United Kingdom}
\email{victor.beresnevich@york.ac.uk}

\author[Erez Nesharim]{Erez Nesharim}
\address{Erez Nesharim, Einstein Institute of Mathematics, The Hebrew University of Jerusalem, Jerusalem, 9190401, Israel}
\email{ereznesh@gmail.com}

\author[Lei Yang]{Lei Yang}
\address{Lei Yang, College of Mathematics, Sichuan University, Chengdu, Sichuan, 610000, China}
\email{lyang861028@gmail.com}

\begin{abstract}
In this paper we prove that badly approximable points on any analytic non-degenerate curve in $\R^n$ is an absolute winning set.
This confirms a key conjecture in the area stated by Badziahin and Velani (2014) which represents a far-reaching generalisation
of Davenport's problem from the 1960s. Amongst various consequences of our main result is a solution to Bugeaud's problem on real
numbers badly approximable by algebraic numbers of arbitrary degree.
The proof relies on new ideas from fractal geometry and homogeneous dynamics.
\end{abstract}

\maketitle

\section{Introduction}\label{sec-introduction}

The primary object of study in this paper will be the sets $\Bad(\rr)$ of {\em $\rr$-badly approximable points} in $\R^n$. Here and elsewhere $\rr=(r_1,\dots,r_n)\in\R^n$ is an $n$-tuple of {\em weights of approximation}, which, by definition, satisfy the following conditions:
\begin{equation}\label{weights}
r_1+\dots+r_n=1\qquad\text{and}\qquad r_i\ge0\quad\forall ~i\,.
\end{equation}
Recall that $\vv x=(x_1,\dots,x_n)\in\Bad(\rr)$ if and only if there is exists $c=c(\vv x)>0$
such that
\begin{equation}\label{eq702}
\max_{1\le i\le n}q^{r_i}|q x_i-p_i|\ge c
\end{equation}
for all $q\in\N$ and $\vv p=(p_1,\dots,p_n)\in\Z^n$. It is well known that when $c=1$ then for all $\vv x\in\R^n$ the opposite to \eqref{eq702} holds for infinitely many $(\vv p,q)\in\Z^n\times\N$; see, for instance, \cite[\S1.4.2]{MR3618787}. Thus, justifying their name, badly approximable points exhibit the worst proximity by rational points $\vv p/q$ up to the constant factor $c>0$. When $\rr=\left(\frac1n,\dots,\frac1n\right)$, we will write $\Bad(n)$ for $\Bad\left(\frac1n,\dots,\frac1n\right)$.

Badly approximable points represent one of the central concepts in Diophantine approximation that has been studied for well over a century. Notably, in the 1980s Dani \cite{MR794799} discovered a neat correspondence between points in $\Bad(n)$ and bounded orbits of diagonal flows on the space of lattices in $\R^{n+1}$. Ever since badly approximable points have also been in the spotlight within homogeneous dynamics.

The holy grail of the theory of badly approximable points is {\em the winning property} of the sets $\Bad(\rr)$ and their restrictions to natural geometric objects such as curves, surfaces and fractals in $\R^n$. The winning property is understood in the sense of Schmidt's game \cite{Schmidt-Game} or its generalisations, which detailed account can be found in \cite{BHNS2018}. Indeed, Schmidt's games have long been the main vehicle for many advances in both Diophantine approximation and homogeneous dynamics including \cite{MR248090, MR980795, MR2191212, MR2520102, MR2480100, Klein-Weiss-Modified-Schmidt, MR2720230, MR2808214, MR2855820, MR2818688, MR2981929, MR3296561, MR3436158, An, MR3474376, MR3733884, MR3826896, MR3910474, MR3987817} amongst many others. More recently, stemming from the breakthrough proof of Schmidt's conjecture \cite{BPV}, new approaches evolved, enabling several major developments. These include the proof of Schmidt's conjecture in higher dimensions \cite{Beresnevich, Yang2019}, resolution to Davenport's problem \cite{MR3231023, Beresnevich, Yang2019} and discoveries on multiplicative approximations \cite{Badziahin-Velani2011, MR3028170}. However, the property of being winning, which is the strongest and thus most desirable, was largely non-attainable by these developments, except in dimension $n=2$ \cite{MR3081541, MR3224831, An, MR3733884}. Indeed, for $n>2$, the most general result, which encompasses previous milestones \cite{Schmidt-Game}, \cite{MR2520102} and \cite{MR2981929}, is the Guan-Yu theorem \cite{MR3910474} that $\Bad(\rr)$ with $r_1=\dots=r_{n-1}\ge r_n$ is hyperplane absolute winning. It is also worth mentioning that Kleinbock and Weiss \cite{Klein-Weiss-Modified-Schmidt} proved that $\Bad(\rr)$ is modified winning, where the notion of winning depends on $\rr$. The goal of this paper is to address the following key challenge.

\begin{prob}[{\cite[Conjecture~B]{MR3231023}}]\label{p1}
Prove that the set of $\rr$-badly approximable points lying on a nondegenerate curve in $\R^n$ is winning.
\end{prob}




Our main result (Theorem~\ref{thm:main-theorem} below) settles this problem in full for analytic nondegenerate curves in arbitrary dimensions. This represents the first result of its kind beyond planar curves. In fact, even for $n=2$ our result is new as we establish the {\em absolute winning} property, which is stronger than Schmidt's winning property that was established in \cite{MR3733884} for all non-degenerate planar curves. As a consequence of our main result, we are able to resolve Bugeaud's problem \cite[Problem~2.9.11]{MR3618788}\footnote{Problem~9.9 in the preprint version \url{arXiv:1502.03052}} 
on real numbers badly approximable by algebraic numbers of arbitrary degree.

Prior to stating our main result, recall that an analytic map $\varphi=(\varphi_1,\dots,\varphi_n):U\to\R^n$, defined on an open interval $U\subset\R$, is {\em nondegenerate} if $\varphi_1,\dots,\varphi_n$ together with the constant function $1$ are linearly independent over $\R$.
Also, recall that $S\subset\R$ is absolute winning on $U$ if and only if $S\cup(\R\setminus U)$ is absolute winning.

\begin{thm}\label{thm:main-theorem}
Let $\rr$ be an $n$-tuple of weights of approximation, $U\subset\R$ be a non-empty open interval and $\varphi:U\to\R^n$ be a nondegenerate analytic map. Then the set $\varphi^{-1}\left(\Bad(\rr)\right)$ is absolute winning on $U$.
\end{thm}

\subsection{Properties of absolute winning sets}

Absolute winning sets were introduced by McMullen in \cite{MR2720230} in order to strengthen the notion of Schmidt's winning sets \cite{Schmidt-Game}. Precise definitions and various properties of (absolute) winning sets can be found in \cite{Schmidt-Game,MR2720230,MR2981929,BHNS2018}. In particular, we have the following lemma that is assembled from known results that can be found in \cite{MR2720230, MR2981929, BHNS2018}. In what follows, $\dim$ stands for Hausdorff dimension.

\begin{lemma}\label{abswinprop} The following properties hold for subsets of $\R$\/{\rm:}\\[-3.5ex]
\begin{enumerate}
  \item[{\rm(i)}] absolute winning sets are preserved by quasi-symmetric homeomorphisms, including by-Lipschitz maps, and $C^1$ diffeomorphisms;
  \item[{\rm(ii)}] any countable intersection of absolute winning sets is absolute winning;
  \item[{\rm(iii)}] any absolute winning set has full Hausdorff dimension; and furthermore
  \item[{\rm(iv)}] for any absolute winning set $S\subset\R$, any open interval $U$ and any Ahlfors regular measure $\mu$ on $\R$ such that $U\cap \supp\mu\neq\emptyset$ we have that
\begin{equation}\label{dim1}
      \dim(S\cap U\cap\supp\mu)=\dim(\supp\mu)\,.
\end{equation}
\end{enumerate}
\end{lemma}

A Borel measure $\mu$ on $\R^d$ will be called {\em $(C,\alpha)$-Ahlfors regular}, where $C >1$ and $\alpha >0$ are constants, if there exists $\rho_0>0$ such that for any ball $B(x,\rho)\subset\R$ centred at $x\in\supp \mu$ of radius $\rho>0$ we have that
\begin{equation}\label{Ahlfors}
C^{-1} \rho^{\alpha} \le \mu(B(x,\rho)) \le C \rho^{\alpha}\qquad\text{if }\rho\le \rho_0\,.
\end{equation}
We say that $\mu$ is {\em $\alpha$-Ahlfors regular}\/ if there exists $C >1$ such that $\mu$ is $(C,\alpha)$-Ahlfors regular; and we say that $\mu$ is {\em Ahlfors regular}\/ if it is $\alpha$-Ahlfors regular for some $\alpha>0$. Recall that $\dim(\supp\mu)=\alpha$ for any $\alpha$-Ahlfors regular measure $\mu$. In particular, $\supp\mu$ has the cardinality of continuum.

Property (iv) of Lemma~\ref{abswinprop} merits a special comment as understanding the intersection of $S=\varphi^{-1}(\Bad(\rr))$ with the supports of Ahlfors regular measures will underpin our approach to establishing Theorem~\ref{thm:main-theorem}. First of all, note that the simplest and indeed weakest meaningful consequence of \eqref{dim1} is that
\begin{equation}\label{dim2}
 S\cap U\cap\supp\mu\neq\emptyset\,.
\end{equation}
However, as it happens, the validity of \eqref{dim2} for every Ahlfors regular measure $\mu$ such that $\supp\mu\cap U\neq\emptyset$ is enough for ensuring the much stronger full-dimension property \eqref{dim1}, and furthermore the absolute winning property of $S$ on $U$. This is a consequence of the following two facts recently established in \cite{BHNS2018}.

\begin{lemma}[Corollary 4.2 in \cite{BHNS2018}]\label{thm:intersection-fractals-winning}
Let $U\subset\R$ be a non-empty open interval and $S \subset U$ be a Borel subset. Suppose that for every diffuse subset $K\subset U$ we have that $S\cap K\neq \emptyset$. Then $S$ is absolute winning on $U$.
\end{lemma}

\begin{lemma}[Proposition 2.18 in \cite{BHNS2018}]\label{diffuse}
Let $U\subset\R$ be a non-empty open interval and $K \subset U$ be a diffuse set. Then $K$ contains a closed subset that supports an Ahlfors regular measure.
\end{lemma}

Thus, in view of Lemmas~\ref{abswinprop}(iv), \ref{thm:intersection-fractals-winning} and \ref{diffuse}, we have that Theorem~\ref{thm:main-theorem} is equivalent to the following statement.

\begin{thm}\label{thm:main-theorem2}
Let $\rr$ be an $n$-tuple of weights of approximation, $U\subset\R$ be an open interval, $\mu$ be an Ahlfors regular measure such that $U\cap\supp\mu\neq\emptyset$ and $\varphi:U\to\R^n$ be a nondegenerate analytic map. Then
\begin{equation}\label{equ:intersection-fractal}
\varphi^{-1}(\Bad(\rr)) \cap \supp \mu \neq \emptyset.
\end{equation}
\end{thm}

To sum up, we have that
\begin{equation}\label{equiv}
\text{Theorem~\ref{thm:main-theorem} holds}\quad\iff\quad\text{Theorem~\ref{thm:main-theorem2} holds.}
\end{equation}

\medskip

\subsection{Corollaries to the main theorem}

We begin with the following straightforward consequence of Theorem~\ref{thm:main-theorem} combined with Lemma~\ref{abswinprop}(ii)--(iv).

\begin{cor}\label{cor3}
Let $U\subset\R$ be an open interval and $\mu$ be an Ahlfors regular measure such that $\supp\mu\cap U\neq\emptyset$.  Further for each $i\in\N$ let $n_i\in\N$ be given, $\rr_i$ be an $n_i$-tuple of weights of approximation and $\vv f_i:U\to\R^{n_i}$ be an analytic nondegenerate map. Then,
$$
\dim\bigcap_{i\in\N}\vv f_i^{-1}(\Bad(\rr_i))\cap\supp\mu=\dim(\supp\mu)\,.
$$
In particular, when $\mu$ is Lebesgue measure, we have that
$$
\dim\bigcap_{i\in\N}\vv f_i^{-1}(\Bad(\rr_i))=1\,.
$$
\end{cor}

Our next corollary resolves a problem on real numbers badly approximable by algebraic numbers of arbitrary degrees. In what follows, given a polynomial $P=a_nx^n+\dots+a_0\in\Z[x]$, $H(P)=\max\{|a_0|,\dots,|a_n|\}$ will denote the height of $P$.
Given an algebraic number $\alpha\in\C$, $H(\alpha)$ will denote the (naive) height of $\alpha$, which, by definition, is the height of the minimal defining polynomial $P$ of $\alpha$ over $\Z$.

To motivate the definition of real numbers badly approximable by algebraic numbers of degree $n$, recall the long standing Wirsing--Schmidt conjecture \cite[p.258]{MR568710} which states that for any $n\in\N$ and any real transcendental number $\xi$ there is a constant $C=C(\xi,n)>0$ such that
\begin{equation}\label{vb1.8}
|\xi-\alpha|\le C(n,\xi) H(\alpha)^{-n-1}
\end{equation}
holds for infinitely many algebraic numbers $\alpha$ of degree $\le n$. The $n=1$ case of the conjecture is a trivial consequence of the theory of continued fractions; for $n=2$ it was proved by Davenport and Schmidt. However, there are only partial results for $n>2$, see \cite{BadzSch}. Furthermore, it was established in \cite{MR1709049} that the Wirsing--Schmidt conjecture is true for almost every $\xi\in\R$ even if the right hand side of \eqref{vb1.8} was replaced by a decreasing function $\Psi$ of $H(\alpha)$ such that $\sum_{h=1}^\infty h^{n-1}\Psi(h)=\infty$, for example, $\Psi(H(\alpha)=H(\alpha)^{-n-1}\log^{-1} H(\alpha)$. Real numbers badly approximable by algebraic numbers of degree $n$ are introduced by reversing \eqref{vb1.8} with a suitably small constant. Namely, we define the set of real numbers badly approximable by algebraic numbers of degree $n$ as follows:
\begin{align*}
\cB_n^*&=\left\{\xi\in\R:\begin{array}{l}
\exists\ c_1=c_1(\xi,n)>0\text{ such that }|\xi-\alpha|\ge c_1H(\alpha)^{-n-1}\\
\text{for all real algebraic }\alpha\text{ with } \deg\alpha\le n
               \end{array}
\right\}.
\end{align*}
Let us also introduce the set of real number for which the Wirsing--Schmidt conjecture holds:
\begin{align*}
\cW_n^*&=\left\{\xi\in\R:\begin{array}{l}
\exists\ c_2=c_2(\xi,n)>0\text{ such that }|\xi-\alpha|<c_2H(\alpha)^{-n-1}\\
\text{for infinitely many real algebraic }\alpha\text{ with }\deg\alpha\le n
               \end{array}
\right\}
\end{align*}
and the following set of real numbers badly approximable in terms of small values of integral polynomials of degree $n$:
\begin{align*}
\cB_n&=\left\{\xi\in\R:\begin{array}{l}
\exists\ c_1=c_1(\xi,n)>0\text{ such that }|P(\xi)|\ge c_1H(P)^{-n}\\
\text{for all non-zero }P\in\Z[x],\ \deg P\le n
               \end{array}
\right\}.
\end{align*}
It is well known, see for example \cite{Beresnevich}, that
\begin{equation}\label{incl}
\cB_n\subset\cW_n^*\cap\cB_n^*.
\end{equation}

The existence of transcendental numbers lying in $\cB_n^*$ and furthermore in $\cB_n$ was conjectured by Bugeaud in \cite[\S10.2]{MR2136100} and established in \cite{Beresnevich}. More precisely, it was proved in \cite{Beresnevich} that the intersection of any finite number of the sets $\cB_n$ with any interval in $\R$ has full Hausdorff dimension. The strong form of Bugeaud's problem claims that the intersection of all the sets $\cB_n$ has full dimension, see \cite[Problem~2.9.11]{MR3618788}.
The following statement resolves this problem in full generality.

\begin{cor}\label{cor2}
Let $U\subset\R$ be any open interval and $\mu$ be any Ahlfors regular measure such that $\supp\mu\cap U\neq\emptyset$. Then,
\begin{equation*}
\dim \bigcap_{n=1}^\infty \cB_n\cap U\cap \supp\mu=\dim \bigcap_{n=1}^\infty \cB_n^*\cap \cW_n^*\cap U\cap \supp\mu=\dim(\supp\mu)\,.
\end{equation*}
In particular,
$$
\dim \bigcap_{n=1}^\infty \cB_n\cap U=\dim \bigcap_{n=1}^\infty \cB_n^*\cap U=1\,.
$$
\end{cor}

\begin{proof}
It is well known that $\cB_n=\vv f_n^{-1}(\Bad(n))$, where $\vv f_n(x)=(x,\dots,x^n)$, see \cite{Beresnevich}. Therefore, in view of \eqref{incl}, Corollary~\ref{cor2} is a special case of Corollary~\ref{cor3}.
\end{proof}

\begin{remark}
Not only does Corollary~\ref{cor2} resolve the strong form of Bugeaud's problem, it also allows one to detect numbers $\xi$ badly approximable by algebraic numbers of arbitrary degree with additional properties. For example, if $\mu$ is the uniform probability measure supported on the middle third Cantor set $\cK_3$, then Corollary~\ref{cor2} implies that $\cK_3$ contains a subset of $\xi$ of Hausdorff dimension $\log2/\log3=\dim(\cK_3)$ such that $\xi\in\cB_n$ for all $n$. Recall that $\xi\in\cK_3$ if and only if the ternary expansion of $\xi$ does not contain the digit $1$.
\end{remark}

\subsection{Outline of new ideas}

First of all, note that establishing the winning property of a set by investigating its intersections with fractals, which enables the reduction of Theorem~\ref{thm:main-theorem} to Theorem~\ref{thm:main-theorem2}, has only existed as a theoretical principle introduced in \cite{BHNS2018}. In this paper we develop an approach, which for the first time implements this principle in practice and thus enables us to resolve outstanding problems.
Our approach has several novel features, which we will try to outline here.

To begin with, note that the techniques developed to date rely on counting arguments. In particular, \cite{Beresnevich} and \cite{Yang2019} utilise lattice point counting together with a linearisation technique to estimate the number of the so-called `dangerous intervals'. As a result, previously developed tools are only suitable for studying badly approximable points in `continuous structures' but not in fractals, ruling out any possibility of using such tools for establishing Theorem~\ref{thm:main-theorem2}. In this paper we use instead a Quantitative Non-Divergence (QND) estimate for fractals derived from the paper \cite{MR2134453} of Kleinbock, Lindenstrauss and Weiss. More precisely, we use a `hybrid' version of the QND estimate appearing as Theorem~\ref{thm:QnD-fractal-local} below, in which Case~(1) is with respect to a given fractal measure $\mu$, while Case (2) is with respect to Lebesgue measure ({\em i.e.} $x$ is not restricted to the support of $\mu$). The required estimate for the number of `dangerous intervals' is deduced from the measure theoretic bound appearing as Case (1) of Theorem~\ref{thm:QnD-fractal-local}. Specifically this deduction is done within \eqref{vb3:14}--\eqref{bound1} and \eqref{vb3:28}--\eqref{vb834} below. However, there is a challenge in using the QND estimate: we have to rule out the algebraic obstructions appearing as Case~(2) of Theorem~\ref{thm:QnD-fractal-local}.

To outline how this is dealt with, first recall that the sets of badly approximable points are routinely studied with the help of the so-called Dani correspondence. This correspondence relates badly approximable points to bounded orbits of lattices $u(\vv x)\Z^{n+1}$, where $u(\vv x)$ is given by \eqref{u()}, under a certain diagonal flow. This paper is no exception and the corresponding flow is denoted by $a(t)$, see \eqref{eq:a(t)} for an explicit definition. In short, a point $\vv x\in\R^n$ is $\vv r$-badly approximable if and only if for some $\eps>0$ the orbit $a(t)u(\vv x)\Z^{n+1}$ is contained in a compact set $K_\eps$, defined by \eqref{K_eps}. However, in order to rule out the algebraic obstructions arising from Case~(2) of Theorem~\ref{thm:QnD-fractal-local}, we set up a second independent action by elements of a different diagonal subgroup, which is denoted by $b(t)$ and defined in \eqref{eq:b(t)}. In short, we allow the compact set $K_\eps$ to expand (that is $\eps$ to shrink) as we act by this second action. This is reflected in \eqref{pr2}, which essentially appears to be a condition defining `dangerous intervals'. The expansion of $K_\eps$ enables us to minimise the loss of $\mu$-mass in the corresponding Cantor set construction, which is defined in \S\ref{sec2.4}. At the same time, we demonstrate that it can be chosen in such a way that the aforementioned algebraic obstructions are ruled out. In fact, the algebraic obstructions at each step of the Cantor-like construction appear to be removed at previous steps of the Cantor set construction. We expect that further study of the action of $b(t)$ will lead to more applications to Diophantine approximation.

\subsection{Notation and conventions}\label{subsec-notation}

Throughout this paper, $\|\cdot\|$ will denote the Euclidean norm. Given $k\in\N$,
$\vv e_1,\dots, \vv e_{k}$ will denote the standard basis of $\R^{k}$. Given two vectors $\vv a=(a_1,\dots,a_k)$ and $\vv b=(b_1,\dots,b_k)\in\R^k$, by $\vv a\cdot\vv b=a_1b_1+\dots+a_kb_k$ we will denote the standard inner product of $\vv a$ and $\vv b$.

Given an integer $r$ such that $1\le r\le k$, as is well known, the collection
\begin{equation}\label{vb11}
\vv e_I:=\vv e_{i_1}\wedge\cdots\wedge\vv e_{i_r},\quad\text{where $I=\{i_1<\ldots<i_r\}$,}
\end{equation}
forms the standard basis of $\bigwedge^r\left(\R^{k}\right)$ -- the $r$th exterior power of $\R^k$. Given a multivector $\vv a_1\wedge\cdots\wedge\vv a_r\in \bigwedge^r\left(\R^{k}\right)$, by $(\vv a_1\wedge\cdots\wedge\vv a_r)_I$ we will denote its $\vv e_I$-coordinate. The inner product on $\R^k$ induces the inner product on $\bigwedge^r\left(\R^{k}\right)$, which we will also denote by `$\,\cdot\,$', so that \eqref{vb11} becomes an orthonormal basis. Further, $\|\vv a_1\wedge\cdots\wedge\vv a_r\|$ will denote the Euclidean norm of $\vv a_1\wedge\cdots\wedge\vv a_r$. Given a $k\times k$ matrix $L$ and a multi-vector $\vv v=\vv v_1\we\cdots\we\vv v_r\in\bigwedge^r\left(\R^k\right)$, we define $L\vv v$ as $L\vv v_1\we\cdots \we L\vv v_r$, where all $\vv v_i$ are viewed as columns. Further,
a collection $\vv v_1,\dots,\vv v_r\in\Z^{k}$ will be called {\em primitive} if $\vv v_1,\dots,\vv v_r$ are linearly independent and
$$
\Span_\Z(\vv v_1,\dots,\vv v_r)=\Z^k\cap\Span_\R(\vv v_1,\dots,\vv v_r).
$$
We will use the following form of the Bachmann--Landau notation. Given a normed space $X$ and $\varepsilon>0$, $O_X(\varepsilon)$ will stand for an element of $X$ which norm is $\le C\varepsilon$ for some constant $C>0$. We will omit the subscript if there is no risk of confusion.

Given a set $S$, $\#S$ will denote the number of elements in $S$. Given an interval $I\subset\R$ and a positive number $\lambda$, $|I|$ will denote the length of $I$, and $\lambda I$ will denote the interval of length $\lambda|I|$ and centred at the same point as $I$.

\subsection*{Acknowledgements.}
LY is supported in part by NSFC grant 11801384 and the Fundamental Research Funds for the
Central Universities YJ201769. LY is also grateful to the University of York for its hospitality during LY's visit when part of this work was completed. EN acknowledges support from ERC 2020 grant
HomDyn (grant no. 833423), ISF grant number  871/17, and ERC 2020 grant HD-App (grant no. 754475).

\section{Intersections with fractals}\label{sec-intersections-fractals}

In view of the equivalence \eqref{equiv}, establishing Theorem~\ref{thm:main-theorem2} will be
the sole goal for the rest of this paper. The proof will be based on finding an appropriate
Cantor set inside the left hand side of \eqref{equ:intersection-fractal}. The goal of
this section is to introduce the construction of such a Cantor set (\S\ref{sec2.4}).
Establishing that it is non-empty will be the subject of subsequent sections. To begin with, we introduce several useful assumptions (\S\ref{prelim}), recall Badziahin-Velani's generalised Cantor sets (\S\ref{sec-cantor-like-construction}) and Dani's correspondence (\S\ref{Danicorr}), and state
quantitative non-divergence estimates (\S\ref{QnD_statements}).

\subsection{Preliminaries}\label{prelim}

From now on we fix any $(C,\alpha)$-Ahlfors regular measure $\mu$ such that $\supp\mu\cap U\neq\emptyset$, where $U$ is an open interval in $\R$ as in Theorem~\ref{thm:main-theorem2}. Throughout, without loss of generality, we will assume that
\begin{equation}\label{wcond}
r_1\ge \ldots\ge r_n>0\,.
\end{equation}
Indeed, if $r_n=0$, then $\Bad(\hat{\rr})\times\R\subset\Bad(\rr)$, where $\hat{\rr}=(r_1,\dots,r_{n-1})$, and thus $\hat\varphi^{-1}(\Bad(\hat{\rr}))\subset\varphi^{-1}(\Bad(\rr))$, where $\hat\varphi=(\varphi_1,\dots,\varphi_{n-1}):U\to\R^{n-1}$ is analytic and non-degenerate. Obviously, \eqref{equ:intersection-fractal} follows from $\hat\varphi^{-1}(\Bad(\hat\rr)) \cap \supp \mu \neq \emptyset$. Thus, if $r_n=0$, we reduce the proof of Theorem~\ref{thm:main-theorem2} to a lower dimension.

Since $\varphi$ is non-degenerate, $\varphi'_1$ is not identically zero. Since $\varphi'_1$ is analytic, it can vanish only at a countable number of points. Since $\supp\mu\cap U\neq\emptyset$ and $U$ is open, the set $\supp\mu\cap U$ is uncountable. Hence, there exists a point $x_0\in \supp\mu\cap U$ such that $\varphi'_1(x_0)\neq0$. In what follows $I_0\subset U$ will be a closed interval centred at $x_0$. Clearly, to establish \eqref{equ:intersection-fractal} it suffices to prove that $\varphi^{-1}(\Bad(\rr)) \cap \supp \mu \cap I_0\neq\emptyset$.
Since $\varphi'_1(x_0)\neq0$, by continuity, $\left|\varphi'_1(x)\right|$ is bounded above and below by positive constants for all $x$ in a neighborhood of $x_0$. Then, we can make a change of the variable $x$ to ensure that $\varphi_1(x)=x$ for all $x$ in this neighborhood. Therefore, by shrinking $I_0$ and making the change of variable if necessary, we can ensure that $3^{n+1}I_0\subset U$ and that
\begin{equation}\label{phi}
\varphi(x) = (x, \varphi_2(x), \dots, \varphi_n(x))\qquad\text{for all $x\in 3^{n+1}I_0\subset U$\,.}
\end{equation}
We will also assume that
\begin{equation}\label{lengthI_0}
3|I_0|\le \rho_0\,,
\end{equation}
where $\rho_0$ is as in \eqref{Ahlfors}. Since $I_0$ is centred at $\supp\mu$, by \eqref{Ahlfors}, we have that
\begin{equation}\label{I_0}
C^{-1}\left(\tfrac12|I_0|\right)^\alpha \le \mu(I_0) \le C\left(\tfrac12|I_0|\right)^\alpha\,.
\end{equation}
Also, observe that for any interval $I\subset I_0$, if $x\in I\cap\supp\mu$ then
\begin{equation}\label{incl1}
I\subset B(x,|I|)\subset 3I\,,
\end{equation}
where $B(x,|I|)$ is the closed ball centred at $x$ of radius $|I|$. Then, by \eqref{Ahlfors} and \eqref{incl1} applied to $B(x,|I|)$, we get that
\begin{equation}\label{Ahlfors2}
  \mu(I)\le C|I|^\alpha\quad\text{ and }\quad\mu(3I)\ge C^{-1}|I|^\alpha
\end{equation}
for any interval
$I\subset I_0$ such that $I\cap\supp\mu\neq\emptyset$.

\subsection{Generalised Cantor sets}\label{sec-cantor-like-construction}

In this subsection we recall a construction of Cantor sets proposed in \cite{Badziahin-Velani2011}. Let $R \ge2$ be an integer. Given any closed bounded interval $I$, $\Par_R(I)$ will denote the collection of closed intervals obtained by dividing $I$ to $R$ closed subintervals of the same length $R^{-1}|I|$. More generally, if we have a finite collection $\cJ$ of closed bounded intervals, we define $\Par_R(\cJ)$ to be the union of $\Par_R(I)$ over all $I\in\cJ$. Thus, to obtain $\Par_R(\cJ)$ we have to sub-divide each interval in $\cJ$ into $R$ closed sub-intervals of equal length.

Now, to begin the construction of a Cantor set, let
$\cJ_0 = \{I_0\}$, where $I_0$ is a given closed bounded interval. Then recurrently for $q=0,1,2,\dots$ we perform two steps:

\begin{itemize}
\item Split each interval in $\cJ_q$ into $R$ closed equal subintervals to obtain $\cI_{q+1}$, {{\em i.e.}}
\begin{equation*}
\cI_{q+1}:=\Par_R\left(\cJ_q\right)\,.
\end{equation*}
\item Remove a sub-collection $\hat\cJ_{q}$ from $\cI_{q+1}$ to obtain
\begin{equation}\label{S2}
\cJ_{q+1}:=\cI_{q+1}\setminus \hat{\cJ}_{q}\,.
\end{equation}
\end{itemize}
Observe that
\begin{equation}\label{lengthIq+1}
\left|I\right|=R^{-q-1}|I_0|\qquad\text{for all $I\in\cI_{q+1}$}
\end{equation}
and that the intervals in $\cI_{q+1}$ can intersect each other only at end-points.

The sequence $(\cJ_q)_{q\ge0}$ defines the limit set
\begin{equation}\label{Kinfty}
\cK_\infty := \bigcap_{q =0}^{\infty} ~\bigcup_{I \in \cJ_q} I\,.
\end{equation}
Now, let $ \bh = \left( \h_{p,q}\right)_{0\le p \le q}$ be a sequence of non-negative integers indexed by $p$ and $q$.
For $q\ge0$ write $\hat{\cJ}_{q}$ as the following union
\begin{equation}\label{union}
\textstyle\hat{\cJ}_{q}=\bigcup_{p=0}^{q} \hat{\cJ}_{p, q}\,.
\end{equation}
If for any $0 \le p \le q$ we have that
\begin{equation}\label{hpqbound}
\#\left\{\Iq\in\hat{\cJ}_{p,q}: \Iq\subset \Jp\right\}\le \h_{p,q}\quad\text{for all $\Jp \in \cJ_p$,}
\end{equation}
then $\left(\cJ_q\right)_{q\ge0}$ will be called an {\em $(R, \bh)$-sequence} and its limit set \eqref{Kinfty}
will be called an {\em $(R, \bh)$-Cantor set}. For obvious reasons, the parameter $R$ will be called the {\em splitting rate} and the quantities $\h_{p,q}$ will be called {\em removal rates} of the construction of $\cK_\infty$.
We will use the following result established in \cite{Badziahin-Velani2011}.

\begin{thm}[Theorem~3 in \cite{Badziahin-Velani2011}]\label{thm:generalized-cantor-set-estimate}
Given an integer $R\ge2$ and\/ a sequence of non-negative integers $\bh = \left( \h_{p,q}\right)_{0\le p \le q}$, let $t_0 = R - \h_{0,0}$ and
\begin{equation}\label{t_q}
t_q := R - \h_{q,q} - \sum_{j=1}^q \frac{\h_{q-j, q}}{\prod_{i=1}^j t_{q-i}}\qquad\text{for }q \ge 1\,.
\end{equation}
Suppose that $t_q >0$ for all $q \ge 0$. Then every $(R, \bh)$-Cantor set is nonempty.
\end{thm}

\subsection{Dani's correspondence}\label{Danicorr}

In what follows $X_{n+1} = \SL{n+1}{\R}/\SL{n+1}{\Z}$ will denote the homogeneous space of unimodular lattices in $\R^{n+1}$. For every $g\in \SL{n+1}{\R}$, the corresponding lattice is given by $\lat{g}\in X_{n+1}$.
As is well known, the space $X_{n+1}$ is not compact. By Mahler's criterion, every compact subset of $X_{n+1}$ is contained in
\begin{equation}\label{K_eps}
K_{\epsilon} := \left\{ \Lambda \in X_{n+1}: \inf_{\vv v\in\Lambda,\,\vv v\neq\vv0}\|\vv v\|\ge\varepsilon \right\}
\end{equation}
for some $\epsilon>0$.
As we have already mentioned, Dani \cite{MR794799} discovered a correspondence between points in $\Bad(n)$ and bounded orbits of diagonal flows on $X_{n+1}$. In this paper we will use a more general version of this correspondence -- Lemma~\ref{prop:dani-corrspondence} below, which was formally established in \cite{MR1646538}. Given $\vv x \in \R^n$, define the matrix
\begin{equation}\label{u()}
u(\vv x):= \begin{bmatrix}1 & \vv x \\ & \I_n \end{bmatrix} \in \SL{n+1}{\R}\,,
\end{equation}
where $\I_n$ stands for the $n\times n$ identity matrix and $\vv x$ is regarded as a row.
Further, for $t \in \R$ define the diagonal matrix
\begin{equation}
\label{eq:a(t)}
a(t):= \diag\left\{e^t, e^{-r_1t}, \dots,e^{-r_nt} \right\}=\begin{bmatrix} e^t & & &  \\ & e^{-r_1t} & &  \\ & & \ddots &  \\ & & & e^{-r_nt}  \end{bmatrix} \in \SL{n+1}{\R}.
\end{equation}

\begin{lemma}\label{prop:dani-corrspondence}
Let $\vv x \in \R^n$. Then, $\vv x\in\Bad(\vv r)$ if and only if
$$
\text{$\left\{ \lat{a(t)u(\vv x)} : t >0 \right\}~\subset~ K_\epsilon$ \quad for some $\epsilon>0$,}
$$
that is $\left\{ \lat{a(t)u(\vv x)} : t >0 \right\}$ is bounded in $X_{n+1}$.
\end{lemma}

To be precise, Lemma~\ref{prop:dani-corrspondence} is a consequence of
Mahler's transference \cite{MR0001241} and Theorem~1.5 in \cite{MR1646538}.
See also \cite[Appendix]{BPV} and \cite[Appendix A]{Beresnevich}.

The proof of our main result will also make use of a linearisation technique on $X_{n+1}$ which will utilize the action on $X_{n+1}$ by
\begin{equation}\label{eq:b(t)}
b(t) := \diag\left\{e^{-t/n},e^{t},e^{-t/n},\dots,e^{-t/n}\right\}\in\SL{n+1}{\R}\,.
\end{equation}
The linearisation technique will also make use of the map $z: I_0 \to \SL{n+1}{\R}$ given by
\begin{equation}
\label{eq:z(x)}
z(x) := u_1(\psi(x))\,,
\end{equation}
where
\begin{equation}\label{eq:psi}
\psi(x) := \left(\varphi'_2(x), \dots, \varphi'_n(x)\right)
\end{equation}
and
\begin{equation*}
  u_1(\vv y) := \begin{bmatrix} 1 & 0 & 0 & \cdots & 0 \\
   & 1 & y_2 & \cdots & y_{n} & \\  & & 1 & & \\ & & & \ddots & \\ & & & & 1 \end{bmatrix} \in \SL{n+1}{\R}
\end{equation*}
for $\vv y = (y_2, \dots, y_n)\in\R^{n-1}$.

Before moving on, we state several conjugation equations, which involve the actions by $a(t)$, $b(t)$ and $z(x)$.

\begin{lemma}\label{LemmaConjugate}
For any $t>0$, $x\in 3^{n+1}I_0$, $\vv x=(x_1,\dots,x_n)\in\R^n$ and $\vv y=(y_2,\dots,y_n)\in\R^{n-1}$ we have that
\begin{align}
a(t) u(\vv x) a(-t) &= u\left(e^{(1+r_1)t} x_1, \dots, e^{(1+r_n)t} x_n\right)\,,\label{eq:conj-u-a}\\[0ex]
a(t) u_1(\vv y) a(-t) &= u_1\left(e^{(r_2-r_1)t} y_2, \dots, e^{(r_n - r_1)t} y_n \right)\,,\label{eq:conj-u1-a}\\[0ex]
b(t) u(\vv x) b(-t) &= u\left(e^{-(1+1/n)t}x_1, x_2, \dots, x_n\right)\,,\label{eq:conj-u-b}\\[0ex]
b(t) u_1(\vv y) b(-t) &= u_1\left(e^{(1+1/n)t} \vv y\right)\,,\label{eq:conj-u1-b}\\[0ex]
z(x) u\left(\varphi'(x)\right) z^{-1}(x) &= u(\vv e_1)\,.\label{conj}
\end{align}
\end{lemma}

The proof of these equations is elementary and obtained by inspecting them one by one. The details are left to the reader.

\subsection{Quantitative non-divergence estimates}\label{QnD_statements}

In this section we state auxiliary tools that will be used to estimate the removal rates $\h_{p,q}$ of the Cantor set that will be constructed in \S\ref{sec2.4}. Let
\begin{equation*}
W(\tau,J,\delta):=\left\{ x \in J: \lat{g_{\vv \tau} z(x) u(\varphi(x))} \not\in K_{\delta} \right\}\,,
\end{equation*}
where $J\subset I_0$ is an interval, $z(x)$ is defined by \eqref{eq:z(x)}, $u(\varphi(x))$ is given by \eqref{u()} and
\begin{equation}\label{gtau}
g_{\tau} := \diag\left\{e^{\tau_1} , e^{\tau_2}, \dots, e^{\tau_{n+1}}\right\} \in \SL{n+1}{\R}
\end{equation}
for some $\tau = (\tau_1, \tau_2, \dots, \tau_{n+1})$. Note that $g_{\tau} \in \SL{n+1}{\R}$ if and only if
\begin{equation}\label{tau}
\tau_1+\dots+\tau_{n+1} =0\,.
\end{equation}

\begin{thm}[Local Estimate]\label{thm:QnD-fractal-local}
Suppose that $\mu$, $\varphi$ and $I_0$ are the same as in {\rm\S\ref{prelim}}.
Then, provided $I_0$ is sufficiently small, there exist constants $M_1 >0$ and $\gamma >0$ such that  any $\delta >0$,
any $\tau = (\tau_1, \tau_2, \dots, \tau_{n+1})$ satisfying \eqref{tau}, any $0\le\rho\le1$ and any subinterval $J \subset I_0$ at least one of the following two conclusions holds\/{\rm:}\\[-2ex]
\begin{enumerate}
\item[$(1)$] $\mu(W(\tau,J,\delta)) \le M_1 \left(\dfrac{\delta}{\rho}\right)^{\gamma} \mu(3J);$\\[-1ex]
\item[$(2)$] there exist $1 \le i \le n$ and $\vv v = \vv v_1 \wedge \cdots \wedge \vv v_i \in \bigwedge^i \left(\Z^{n+1}\right) \setminus \{ \vv 0\}$ such that
$$
\sup_{x \in J} \|g_{\vv \tau} z(x) u(\varphi(x)) \vv v\| < \rho\,.
$$
\end{enumerate}
\end{thm}


\begin{thm}[Global Estimate]\label{thm:QnD-fractal-global}
Suppose that $\mu$, $\varphi$ and $I_0$ are the same as in {\rm\S\ref{prelim}}.
Then, provided $I_0$ is sufficiently small, there exist constants $M_2 >0$ and $\gamma >0$ such that for any $\delta >0$ and any
$\tau = (\tau_1, \tau_2, \dots, \tau_{n+1})$ satisfying \eqref{tau} such that
\begin{equation}\label{tau_cond}
\tau_1 >0\qquad\text{and}\qquad\tau_i <0\quad\text{for $i=3,\dots, n+1$}
\end{equation}
one has that
$
\mu(W(\tau,I_0,\delta)) \le M_2 \delta^{\gamma}.
$
\end{thm}

The proof of these theorems is independent from the rest of this paper and is therefore deferred till \S\ref{sec-QnD} in order to maintain the flow of the proof of the main theorem.

\subsection{The construction of a suitable Cantor set}\label{sec2.4}

First fix a large enough integer $R$ and two auxiliary parameters $\beta, \beta' >1$ defined by
\begin{equation}\label{betas}
e^{(1+r_1)\beta} = R\qquad\text{and}\qquad e^{(1+1/n)\beta'} = R.
\end{equation}
Note that since $r_1\ge 1/n$, we have that $\beta\le \beta'$. As
in \S\ref{sec-cantor-like-construction}, $R$ will be
the splitting rate of our construction of a Cantor set.
The parameters $\beta$ and $\beta'$ will determine
the `speed of travel' through $a(t)$- and $b(t)$-orbits in relation
to the Cantor set construction -- see \eqref{pr2}. Next, fix any
\begin{equation}\label{eq:eps}
0 < \epsilon \leq \frac{r_n}{3n}\,.
\end{equation}

As in \S\ref{sec-cantor-like-construction}, to begin the construction of a Cantor set,
we let $\cJ_0 =\{I_0\}$, where $I_0\subset U$ is as in \S\ref{prelim}
and is sufficiently small so that Theorems~\ref{thm:QnD-fractal-local} and \ref{thm:QnD-fractal-global} are applicable. In particular, \eqref{phi}, \eqref{lengthI_0}, \eqref{I_0} and \eqref{Ahlfors2} are satisfied.
Subsequently, for $q=0,1,\dots$ in order to define $\cJ_{q+1}$ we are required to specify the collections $\hat{\cJ}_{q}$ of the removed intervals at each step -- see \eqref{S2}.
In view of \eqref{union}, this will be accomplished if we specify the
components $\hat{\cJ}_{p,q}$ of $\hat{\cJ}_{q}$. With this in mind,
let us assume that $q \ge 0$ and $\cJ_0,\dots,\cJ_q$ have been constructed.

First, for $p=q$ define
\begin{equation}\label{Jqq}
\hat{\cJ}_{q,q}=\left\{\Iq \in \cI_{q+1} ~:~ \mu(\Iq)<(3C)^{-1}|\Iq|^\alpha\right\}.
\end{equation}


Next, for any $q\geq1$ define $\hat{\cJ}_{p,q} = \emptyset$ if $0<p\le q/2$ or $0<p<q$ with $p \not\equiv q \pmod4$, define $\hat{\cJ}_{0,q}$ to be the collection of $\Iq \in \cI_{q+1}\setminus \hat{\cJ}_{q,q}$ such that there exists $l\in\Z$ with
$\max(1,\etaqt) \le l \le \etaq$
satisfying
\begin{equation}\label{pr2}
\lat{b\left(\beta' l\right)  a(\beta(q+1)) z(x) u(\varphi(x))} \not\in K_{e^{-\epsilon\beta l}}\quad\text{for some $x \in \Iq$}\,,
\end{equation}
and finally if $q/2 < p < q$ and $p=q-4l$ for some $l\in\Z$ define
$$
\textstyle \hat{\cJ}_{p,q}:=\left\{\Iq \in \cI_{q+1}\setminus \left(\hat{\cJ}_{q,q}\cup\bigcup_{0\le p'<p}\hat{\cJ}_{p',q}\right)
\,:\,\text{\eqref{pr2} holds}\right\}\,.
$$
Note that in the latter case $1 \leq l < \etaqt$.
This completes the definition of the sets $\hat{\cJ}_{p,q}$ for all possible choices of $0\leq p\leq q$. As in \S\ref{sec-cantor-like-construction}, $\cK_{\infty}$ will be given by \eqref{Kinfty}.

\begin{prop}\label{prop2.3}
With reference to the above construction, we have that
\begin{equation}\label{LB}
  \mu(\Jq)\ge (3C)^{-1}|\Jq|^\alpha
\end{equation}
for all $\Jq\in\cJ_{q+1}$ and $q\ge0$, and
\begin{equation}\label{conclusion1}
\cK_{\infty} \subset \supp\mu\qquad\text{and}\qquad \cK_{\infty} \subset \varphi^{-1}(\Bad(\vv r))\,.
\end{equation}
\end{prop}

\begin{proof}
Regarding \eqref{LB}, this follows from \eqref{S2}, \eqref{union} and \eqref{Jqq}.
Regarding \eqref{conclusion1}, the first inclusion is a consequence of \eqref{LB}, which implies that $\Jqo\cap\supp\mu\neq\emptyset$ for all intervals $\Jqo$ in the construction of the Cantor set $\cK_\infty$.
Next, note that by \eqref{wcond}, \eqref{eq:z(x)} and \eqref{eq:conj-u1-a},
we have that $a(\beta(q+1)) z(x) = u_1(O(1)) a(\beta(q+1))$. So,
\begin{equation}\label{vb498}
b\left(\beta' l\right) a(\beta(q+1)) z(x) u(\varphi(x)) = b\left(\beta' l\right) u_1(O(1)) a(\beta(q+1)) u(\varphi(x)).
\end{equation}
Now, by definition, for all $q \ge \etainvt$, for all $\Iq\in\cJ_{q+1}$ and all $l$ such that $\m \le l \le \etaqt$ condition \eqref{pr2} does not hold. Therefore, using \eqref{vb498} with $l=\m$ we obtain that for any
$q\ge \etainvt\mo$, $\Iq\in\cJ_{q+1}$ and any $x\in \Iq$ we have that
\begin{equation}\label{vb409}
a(\beta(q+1)) u(\varphi(x))\in u_1(O(1))^{-1} b\left(-\beta'\mo\right) K_{e^{-\epsilon\beta \mo}}\,.
\end{equation}
In particular, \eqref{vb409} holds for any $x\in \cK_\infty$ and all $q\ge \etainvt\mo$.
It is readily seen that the right hand side of \eqref{vb409} is contained in a bounded subset of $X_{n+1}$ independent of $q$ and $x$. Therefore, $\left\{ \lat{a(\beta q)u(\varphi(x))} : q\in\N\right\}$ is bounded in $X_{n+1}$, and hence the orbit
$\left\{ \lat{a(t)u(\varphi(x))} : t >0 \right\}$ is bounded in $X_{n+1}$.
By Lemma~\ref{prop:dani-corrspondence}, this shows that $\varphi(x)\in\Bad(\vv r)$. Since this is true for any $x\in \cK_\infty$, the right hand side of \eqref{conclusion1} follows.
\end{proof}

In view of \eqref{conclusion1}, our main task will be to prove that $\cK_{\infty}\neq\emptyset$. This will be done by making use of Theorem~\ref{thm:generalized-cantor-set-estimate}. Naturally, to accomplish this goal we will need to estimate the
{\em removal rates} $\h_{p,q}$. We end this section by doing this when $p=q$.

\begin{prop}\label{Prop_p=q}
Suppose that $R^{\alpha}\ge 21 C^2$.
Then for every $q\ge0$ we have that \eqref{hpqbound} holds for $p=q$ with
\begin{equation}\label{hqq}
\h_{q,q} \le R - (4C)^{-2}R^{\alpha}\,.
\end{equation}
\end{prop}

\begin{proof}
Let $q\ge0$ and take any $\Jq\in\cJ_q$. Then, $\mu(\Jq)\ge (3C)^{-1}|\Jq|^\alpha$. This follows from \eqref{I_0} when $q=0$ and from Proposition~\ref{prop2.3} when $q>0$. Hence, using the left hand side of \eqref{Ahlfors2} and the fact that $|\Iq|=R^{-1}|\Jq|$ for every $I\in\Par_R(\Jq)$ gives
\begin{equation}\label{vb123}
(3C)^{-1}|\Jq|^\alpha\le\mu(\Jq)~=\!\! \sum_{\substack{\Iq\in\Par_R(\Jq)\\[0.2ex]\Iq\cap\supp\mu\neq\emptyset}}\mu(\Iq)~\stackrel{\eqref{Ahlfors2}}{\le}~ N_{\Jq}\times C \left(R^{-1}|\Jq|\right)^{\alpha}\,,
\end{equation}
where $N_{\Jq}$ is the number of summands above. Therefore, \eqref{vb123} implies that $N_{\Jq}\ge \tfrac13C^{-2}R^{\alpha}$.
If $\Iq\in\Par_R(\Jq)$ and $\Iq\cap\supp\mu\neq\emptyset$ and is not the leftmost or rightmost interval of $\Par_R\left(\Jq\right)$, let $\Iq^-\in \Par_R(\Jq)$ and $\Iq^+\in \Par_R(\Jq)$ be the adjacent intervals to the left and to the right of $\Iq$. Thus, $3\Iq=\Iq^-\cup \Iq\cup \Iq^+\subset \Jq$. By the right hand side of \eqref{Ahlfors2}, at least one of these 3 intervals satisfies \eqref{LB}, in particular, the one of the three intervals $\Iq^-$, $\Iq$, $\Iq^+$ that has the largest $\mu$-measure will satisfy \eqref{LB}. Hence, at least
$$
\left(N_{\Jq}-2\right)/3\ge \left(\tfrac13C^{-2}R^{\alpha}-2\right)/3>(3C)^{-2}R^{\alpha} - 1\ge (4C)^{-2}R^{\alpha}
$$
intervals $\Iq\in\Par_R(\Jq)$ will satisfy \eqref{LB}. This immediately implies \eqref{hpqbound} for $p=q$ with $\h_{q,q}$ given by \eqref{hqq}.
\end{proof}

\section{Estimates for removal rates}\label{the_proof}

Within this section we estimate the removal rates $\h_{p,q}$ for $p<q$.
We separately consider the case $p=0$ (Proposition~\ref{prop:estimate-extremely-dangerous})
and $p>0$ (Proposition~\ref{prop:estimate-dangerous}).
Before we proceed we shall prove two auxiliary statements.

\begin{lemma}\label{extension2}
There exists $R_0$ such that for all $R\ge R_0$, $q\geq0$, $0\le l,l'\le q/2$,
$\xzero,\justx \in I_0$ such that $\justx=\xzero+\theta R^{-q-1+l'}$ for some $|\theta|\leq |I_0|$ and any $\vv v=\vv v_1\we\cdots\we\vv v_i\in\bigwedge^i\left(\R^{n+1}\right)\setminus\{\vv0\}$ we have that
\begin{equation}\label{vb590A}
\frac12\le\frac{\left\|u\left(\theta R^{l'-l} \vv e_1\right) H(\xzero)\vv v\right\|}
{\|H(\justx)\vv v\|}\le 2\,,
\end{equation}
where
\begin{equation}\label{H}
H(x)=H_{l,q}(x):=b\left(\beta' l\right)a(\beta(q+1))z(x)u(\varphi(x))\,.
\end{equation}
\end{lemma}

\begin{proof}
Using Taylor's expansion and the well-known and easily verified equations
$$
  u(\vv x+\tilde{\vv x})=u(\vv x)u(\tilde{\vv x})\qquad\text{and}\qquad
  u_1(\vv y+\tilde{\vv y})=u_1(\vv y)u_1(\tilde{\vv y})
$$
valid for all $\vv x,\tilde{\vv x}\in\R^n$ and all $\vv y,\tilde{\vv y}\in\R^{n-1}$, we obtain that
\begin{align}
  u(\varphi(\justx)) &= u\left(\varphi(\xzero) + \theta R^{-q-1+l'} \varphi'(\xzero) + O\left(R^{-2q-2+2l'}\right)\right) \nonumber \\
                 &= u\left( \theta R^{-q-1+l'} \varphi'(\xzero) + O\left(R^{-2q-2+2l'}\right)\right) u(\varphi(\xzero))\,, \label{eq:ly8A}
\end{align}
where $O\left(R^{-2q-2+2l'}\right)=O_{\R^n}\left(R^{-2q-2+2l'}\right)$ -- see \S\ref{subsec-notation} for notation. Similarly,
\begin{align}
  z(\justx) &\stackrel{\eqref{eq:psi}}{=} u_1(\psi(\justx)) = u_1\left(\psi(\xzero) + O\left(R^{-q-1+l'}\right)\right) \nonumber\\
        &= u_1\left(O\left(R^{-q-1+l'}\right)\right) u_1(\psi(\xzero)) \stackrel{\eqref{eq:psi}}{=} u_1\left(O\left(R^{-q-1+l'}\right)\right) z(\xzero)\,, \label{eq:ly9A}
\end{align}
where $O\left(R^{-q-1+l'}\right)=O_{\R^{n-1}}\left(R^{-q-1+l'}\right)$.
Using \eqref{eq:ly9A} together with \eqref{wcond}, \eqref{eq:conj-u1-a}, \eqref{eq:conj-u1-b} and \eqref{betas}, we find that
\begin{align}
        b\left(\beta' l\right)  a(\beta(q+1)) &z(\justx)
        \stackrel{\eqref{eq:ly9A}}{=} b\left(\beta' l\right) a(\beta(q+1)) u_1\left(O\left(R^{-q-1+l'}\right)\right) z(\xzero)   \nonumber\\
&\stackrel{\eqref{wcond}\&\eqref{eq:conj-u1-a}}{=} b\left(\beta' l\right)  u_1\left(O\left(R^{-q-1+l'}\right)\right) a(\beta(q+1)) z(\xzero)  \nonumber\\
        &\stackrel{\eqref{eq:conj-u1-b}\&\eqref{betas}}{=} u_1\left(O\left(R^{-q-1+l'+l}\right)\right) b\left(\beta' l\right) a(\beta(q+1)) z(\xzero)\,. \label{vb968A}
\end{align}
Further, using \eqref{eq:ly8A} together with \eqref{eq:conj-u-a}, \eqref{eq:conj-u-b} and \eqref{conj} we find that
\begin{align}
        &b\left(\beta' l\right) a(\beta(q+1)) z(\xzero) u(\varphi(\justx)) \nonumber\\
&\stackrel{\eqref{eq:ly8A}}{=}  b\left(\beta' l\right)  a(\beta(q+1)) z(\xzero) u\left(\theta R^{-q-1+l'} \varphi'(\xzero) +O\left(R^{-2q-2+2l'}\right)\right) u(\varphi(\xzero)) \nonumber\\
        &\stackrel{\eqref{conj}}{=}  b\left(\beta' l\right) a(\beta(q+1)) u\left(\theta R^{-q-1+l'} \vv e_1 +O\left(R^{-2q-2+2l'}\right)\right)  z(\xzero) u(\varphi(\xzero))\nonumber \\
        &\stackrel{\eqref{eq:conj-u-a}}{=} b\left(\beta' l\right)  u\left( \theta R^{l'}\vv e_1 + O\left(R^{-q-1+2l'}\right)\right) a(\beta(q+1)) z(\xzero) u(\varphi(\xzero))\nonumber \\
        &\stackrel{\eqref{eq:conj-u-b}\&\eqref{betas}}{=} u\left(O\left(R^{-q-1+2l'}\right)\right) u\left(\theta R^{l'-l}  \vv e_1 \right) b\left(\beta' l\right) a(\beta(q+1)) z(\xzero) u(\varphi(\xzero))\,. \label{vb386A}
\end{align}
Combining this with \eqref{vb968A} gives that
\begin{align}
        H(\justx)&\hspace*{1ex}=\hspace*{1ex}b\left(\beta' l\right)  a(\beta(q+1)) z(\justx)  u(\varphi(\justx))\nonumber\\
        &\stackrel{\eqref{vb968A}}{=} u_1\left(O\left(R^{-q-1+l'+l}\right)\right) b\left(\beta' l\right) a(\beta(q+1)) z(\xzero)u(\varphi(\justx)) \nonumber\\
        &\stackrel{\eqref{vb386A}}{=}O_e\left(R^{-1}\right) u\left(\theta R^{l'-l}  \vv e_1 \right) b\left(\beta' l\right) a(\beta(q+1)) z(\xzero) u(\varphi(\xzero))\nonumber\\
        &\hspace*{1ex}=\hspace*{1ex}O_e\left(R^{-1}\right) u\left(\theta R^{l'-l}  \vv e_1 \right) H(\xzero)\,,\label{vb789A}
\end{align}
where
\begin{equation*}
O_e\left(R^{-1}\right)=u_1\left(O\left(R^{-q-1+l'+l}\right)\right) u\left(O\left(R^{-q-1+2l'}\right)\right)\,.
\end{equation*}
Since $0\le l,l'\le q/2$, the matrix $O_e\left(R^{-1}\right)$
gets within $\SL{n+1}{\R}$ arbitrarily close to the identity matrix if $R\ge R_0$ and $R_0$ is sufficiently large. In this case \eqref{vb789A} implies \eqref{vb590A}, thus completing the proof.
\end{proof}

\begin{lemma}\label{extension}
There exists $R_0'$ such that for all $R\ge R_0'$, $q\geq0$, $\m\le l\le \etaq$ and any interval $\Iq\subset I_0$ of length $|I_0|R^{-q-1}$ for which there exists $\xzero \in \Iq$ satisfying
\begin{equation}\label{vb298}
\lat{H(\xzero)} \not\in K_{e^{-\epsilon\beta l}}
\end{equation}
we have that
\begin{equation}\label{vb590}
\lat{H(\justx)} \not\in K_{3 e^{-\epsilon \beta l}}\quad\text{for all $\justx\in \Iq$}\,.
\end{equation}
\end{lemma}

\begin{proof}
By \eqref{vb298}, there exists $\vv v\in\Z^{n+1}\setminus\{\vv0\}$ such that
\begin{equation}\label{vb875}
\|H(\xzero)\vv v\|<e^{-\epsilon\beta l}\,.
\end{equation}
Since $l\geq \m$ there exists $R_0'$ such that for any $R\geq R_0'$ the operator norm of $u\left(-\theta R^{-l} \vv e_1\right)$ is less than $3/2$ for any $\theta$ satisfying $|\theta|\leq|I_0|$. Hence, for any such $\theta$, by \eqref{vb875}, we get that
\begin{equation}\label{vb9878}
\left\|u\left(\theta R^{-l} \vv e_1\right) H(\xzero)\vv v\right\|< \tfrac32e^{-\epsilon \beta l}\,.
\end{equation}
Take any point $\justx\in \Iq$. Then, $\justx=\xzero+\theta R^{-q-1}$ for some $\theta$ with $|\theta|\le |I_0|$. Let $R_0'\ge R_0$, where $R_0$ arises from Lemma~\ref{extension2}. Note that Lemma~\ref{extension2} is applicable with $l'=0$, in which case $u\left(\theta R^{l'-l} \vv e_1\right)=u\left(\theta R^{-l} \vv e_1\right)$. Then, the left hand side of \eqref{vb590A} together with \eqref{vb9878} imply \eqref{vb590}, as required.
\end{proof}

\begin{prop}\label{prop:estimate-extremely-dangerous}
There exist constants $R_1\geq1$, $C_1 >0$ and $\eta_1>0$ such that if $R\geq R_1$ then, with reference to the Cantor set defined in \S\ref{sec2.4}, for any $q >0$ we have
that \eqref{hpqbound} holds for $p=0$ with
\begin{equation}\label{est1}
\h_{0,q}\le C_1 R^{\alpha(1-\eta_1) (q+1) }\,.
\end{equation}
\end{prop}

\begin{proof}
Recall from \S\ref{sec2.4} that we only need to verify \eqref{est1} for $q\ge \etainv \mo$ since otherwise $\hat{\cJ}_{0,q} = \emptyset$. Let $ l_{\min} :=\max (\m, \etaqt) \le l \le \etaq$, and $N_l$ be the number of
intervals $\Iq \in \cI_{q+1}=\Par_R\left(\cJ_q\right)$ such that $\Iq \not\in \hat{\cJ}_{q,q}$ and \eqref{pr2} holds. Let $\Iq$ be any of these intervals. Since $\Iq \not\in \hat{\cJ}_{q,q}$, by \eqref{lengthIq+1} and \eqref{Jqq}, we have that
\begin{equation}\label{LB2}
\mu(\Iq)\ge (3C)^{-1}|\Iq|^\alpha=(3C)^{-1}|I_0|^\alpha R^{-\alpha(q+1)}\,.
\end{equation}
Furthermore, by \eqref{pr2} and Lemma~\ref{extension}, we have that \eqref{vb590} holds for any $\justx\in \Iq$.
Let $\tau = (\tau_1, \tau_2, \dots, \tau_{n+1})$ be such that $g_{\vv \tau}:=b\left(\beta' l\right) a(\beta(q+1))$ with $g_\tau$ given by \eqref{gtau}. It is readily seen that conditions \eqref{tau} and \eqref{tau_cond} are satisfied. Then, by Theorem~\ref{thm:QnD-fractal-global}, we have that
\begin{equation}\label{vb3:14}
\mu \left(\left\{ \justx \in I_0 : \lat{H(\justx)} \not\in K_{3 e^{ - \epsilon\beta l}} \right\} \right) \le M_2 3^{\gamma} e^{-\gamma\epsilon\beta l}\,.
\end{equation}
On the other hand, by \eqref{LB2} and Lemma~\ref{extension}, this measure is greater than or equal to $N_l\times (3C)^{-1}|I_0|^\alpha R^{-\alpha(q+1)}$. Hence, using \eqref{betas} gives
\begin{equation}\label{bound1}
N_l\le M_2 3^{\gamma} e^{ - \gamma\epsilon\beta l} \times 3C|I_0|^{-\alpha} R^{\alpha(q+1)}~=~
3CM_2 3^{\gamma} |I_0|^{-\alpha} R^{\alpha(q+1) - \frac{\gamma\epsilon}{1+r_1} l}\,.
\end{equation}
Let
$\eta'_1 := \frac{\gamma\epsilon}{1+r_1}$.
Summing up \eqref{bound1} over $l=l_{\min}, \ldots,  \etaq$ gives
\begin{align}
\sharp \hat{\cJ}_{0,q} &\le 3CM_2 3^{\gamma} |I_0|^{-\alpha}  \sum_{l=l_{\min}}^{\etaq}  R^{\alpha (q+1) - \eta'_1 l} \nonumber\\
& = 3CM_2 3^{\gamma} |I_0|^{-\alpha}  R^{\alpha (q+1) - \eta'_1 l_{\min}} \sum_{l=l_{\min}}^{\etaq} R^{-\eta'_1 (l-l_{\min})}\,.\label{vb598}
\end{align}
If $R$ is sufficiently large then
$$
\sum_{l=l_{\min}}^{\etaq} R^{-\eta'_1 (l-l_{\min})} \le \sum_{i=0}^{\infty} R^{-\eta'_1 i} =\frac{1}{1-R^{-\eta'_1}}\le 2\,.
$$
Also recall that $l_{\min} \ge \etaqt$. Therefore, by \eqref{vb598}, we get the desired upper bound with $\eta_1=\tfrac{\eta_1'}{8\alpha}$ and $C_1=6CM_2 3^{\gamma} |I_0|^{-\alpha}$. This completes the proof of Proposition~\ref{prop:estimate-extremely-dangerous}.
\end{proof}

\begin{prop}\label{prop:estimate-dangerous}
There exist constants $R_2\geq1$, $C_2 >0$ and $\eta_2>0$ such that if $R\geq R_2$ then, with reference to the Cantor set defined in \S\ref{sec2.4}, for any $q >0$ and $0<p<q$ we have that \eqref{hpqbound} holds with
$$
\h_{p,q}\le C_2 R^{\alpha(1-\eta_2) (q+1-p)}\,.
$$
\end{prop}

The proof will make use of the following generalisation of Lemma~5.8 in \cite{Yang2019}.

\begin{lemma}\label{lm:yang-lm5.8}
Let $W$ be the subspace of\/ $\R^{n+1}$ spanned by the basis vectors
$\vv e_2, \dots, \vv e_{n+1}$, $\rho>0$, $L\ge\red{1}$, $1\le i\le n$ and $\vv a_1, \dots, \vv a_i \in \R^{n+1}$ be such that
\begin{equation}\label{cond_lem}
  \|u(\Theta \vv e_1) (\vv a_1 \wedge \cdots \wedge \vv a_i)\| \le \rho^i
\end{equation}
for all $\Theta \in [0, L]$. Then at least one of the following two statements holds:\\[-1ex]
\begin{enumerate}
\item[{\rm \caseA{}}] there exists $\vv a =(a_{1}, a_{2}, \dots, a_{n+1})\in\Span_\Z(\vv a_1, \ldots,\vv a_i)$ such that $\|\vv a\|\le\rho$ and $|a_{2}| \le  \rho L^{-\frac{1}{2}};$

\medskip

\item[{\rm \caseB{}}] $i\ge2$ and there exist $\vv w^{(i-1)} \in \bigwedge\nolimits^{i-1} W$ and $\vv w^{(i)} \in \bigwedge^i W$
such that
\begin{equation}\label{vb640}
\vv a_1 \wedge \cdots \wedge \vv a_i = \vv e_1 \wedge \vv w^{(i-1)} + \vv w^{(i)}\,,
\end{equation}
\begin{equation}\label{vb641}
\|\vv w^{(i-1)}\| \le \rho^i\quad\text{and}\quad\|\vv w^{(i)}\| \le \red{4\sqrt n}\rho^{i} L^{-\frac{1}{2}}\,.
\end{equation}
\end{enumerate}
\end{lemma}

\begin{proof}
For convenience, let $\vv a^{(i)} = \vv a_1 \wedge \cdots \wedge \vv a_i$. By \eqref{cond_lem} with $\Theta=0$, we have that
\begin{equation}\label{cond_2}
\|\vv a^{(i)}\| \le \rho^i.
\end{equation}
We will assume that $\vv a^{(i)}\neq\vv0$ as otherwise \caseB{} holds with $\vv w^{(i-1)}=\vv0$ and $\vv w^{(i)}=\vv 0$. By \eqref{cond_2}, the sublattice $\Lambda=\Span_\Z(\vv a_1, \ldots,\vv a_i)$ of $\Span_\R(\vv a_1, \ldots,\vv a_i)$ has determinant $\le\rho^i$. By Minkowski's theorem for convex bodies, $\Lambda$ contains a non-zero vector of length $\le\red{\sqrt i}\rho$. Without loss of generality, we can assume that
\begin{equation}\label{a1}
\|\vv a_1\| \le \red{\sqrt i}\rho
\end{equation}
as otherwise we can replace $\vv a_1,\ldots,\vv a_i$ with any reduced basis of $\Lambda$.
Note that $\vv a^{(i)}$ does not depend on the choice of the basis of $\Lambda$.

First of all, let us deal with the case $i=1$. Let $\vv a=\vv a_1$. Then, by \eqref{a1}, $\|\vv a\|\le\rho$. Further,
computing the action on $\vv a$ by $u\left(\Theta\vv e_1\right)$ gives that
$u\left(\Theta\vv e_1\right)\vv a = \vv a + \Theta a_2 \vv e_1$.
Then, by \eqref{cond_lem}, we have that $\left\|\vv a + \Theta a_2 \vv e_1\right\| \le \rho^i=\rho$ for all $\Theta\in \left[0,L\right]$.
For $\Theta=L$ this implies that $\left|a_1+a_2L\right| \le \rho$. By the triangle inequality, $L\left|a_2\right| \le \rho+\left|a_1\right|\le \rho+\left\|\vv a\right\|\le 2\rho$ and we finally conclude that
$|a_2|\le 2\rho L^{-1}\le \rho L^{-\frac12}$. Thus \caseA{} always holds in the case $i=1$.

Now suppose that $i\ge2$. Let $\tilde{W}$ denote the subspace of $\R^{n+1}$ spanned by $\{\vv e_3, \dots, \vv e_{n+1}\}$.
For $j=1, \dots, i$, write $\vv a_j = a_{j,1} \vv e_1 + a_{j,2} \vv e_2 +  \tilde{\vv w}_j$,
where $\tilde{\vv w}_j \in \tilde W$. For the rest of the proof we will assume that $|a_{1,2}| > \rho L^{-\frac{1}{2}}$ as otherwise \caseA{} holds with $\vv a=\vv a_1$ and we are done. Next, using the basic properties of the wedge product we get that
\begin{align}
\vv a^{(i)} &= \left( a_{1,1} \vv e_1 + a_{1,2} \vv e_2 + \tilde{\vv w}_1\right) \wedge \cdots \wedge \left( a_{i,1} \vv e_1 + a_{i,2} \vv e_2 +  \tilde{\vv w}_i\right) \nonumber \\
&= \vv e_1 \wedge \vv e_2 \wedge \tilde{\vv w}^{(i-2)}  + \vv e_1 \wedge \tilde{\vv w}^{(i-1)}_1
+ \vv e_2 \wedge  \tilde{\vv w}^{(i-1)}_2 + \tilde{\vv w}^{(i)}\,, \label{eq:ly11}
\end{align}
where
\begin{align*}
\tilde{\vv w}^{(i-2)} &= \sum_{1\le j\neq\ell\le i}(-1)^{j+\ell+1}a_{j,1}a_{\ell,2} \bigwedge_{j' \neq j,\ell} \tilde{\vv w}_{j'} \in \bigwedge\nolimits^{i-2}\tilde{W}\,,\\[0ex]
\tilde{\vv w}^{(i-1)}_k &= \sum_{j=1}^i (-1)^{j+1} a_{j,k} \bigwedge_{j' \neq j} \tilde{\vv w}_{j'} \in \bigwedge\nolimits^{i-1}\tilde{W}\qquad\quad(k=1,2),\\[0ex]
\tilde{\vv w}^{(i)} &= \bigwedge_{j=1}^i \tilde{\vv w}_j \in \bigwedge\nolimits^i \tilde{W}\,.
\end{align*}
Let $\vv w^{(i-1)} = \tilde{\vv w}^{(i-1)}_1 + \vv e_2 \wedge \tilde{\vv w}^{(i-2)}$ and $\vv w^{(i)} = \vv e_2 \wedge \tilde{\vv w}_2^{(i-1)} + \tilde{\vv w}^{(i)}$. It is easily seen that $\vv w^{(i-1)}\in \bigwedge^{i-1}W$ and $\vv w^{(i)}\in \bigwedge^{i}W$. Further, by \eqref{eq:ly11}, \eqref{vb640} holds.
Since \eqref{vb11} is an orthonormal basis, the four multivectors in the sum \eqref{eq:ly11} are orthogonal.
Therefore, \eqref{cond_2} immediately implies the left hand side of \eqref{vb641}.

To prove the right hand side of \eqref{vb641}, first note that $u(\Theta \vv e_1)\vv e_1 = \vv e_1$, $u(\Theta \vv e_1)\vv e_2 = \vv e_2 + \Theta \vv e_1 $
and $u(\Theta\vv e_1)\tilde{\vv w} = \tilde{\vv w} $ for any $\tilde{\vv w} \in \tilde{W}$ and any $\Theta \in \R$. Then
\begin{equation}\label{vb296}
u(\Theta \vv e_1) \vv a^{(i)} = \vv a^{(i)} + \Theta \vv e_1 \wedge \tilde{\vv w}^{(i-1)}_2\,.
\end{equation}
Using \eqref{cond_lem} with $\Theta=L$, \eqref{cond_2} and \eqref{vb296} gives
\begin{equation}\label{vb7345}
L\|\vv e_1 \wedge \tilde{\vv w}^{(i-1)}_2\|\le2\rho^i\,.
\end{equation}
Observe that $\|\vv e_1 \wedge \tilde{\vv w}^{(i-1)}_2\|=\|\tilde{\vv w}^{(i-1)}_2\|=\left\|\vv e_2 \wedge \tilde{\vv w}^{(i-1)}_2\right\|$. This together with \eqref{vb7345} gives
\begin{equation}\label{eq:ly13}
\left\|\tilde{\vv w}^{(i-1)}_2\right\|=\left\|\vv e_2 \wedge \tilde{\vv w}^{(i-1)}_2\right\| \le 2\rho^i L^{-1}\,.
\end{equation}
Recall that
\begin{equation}\label{submult}
  \|\vv u\we\vv v\|\le \|\vv u\|\cdot\|\vv v\|\qquad\text{for all $\vv u\in\R^{n+1}$ and all $\vv v\in\bigwedge\nolimits^{i-1}\left(\R^{n+1}\right)$}.
\end{equation}
By \eqref{a1}, we have that $\left\|\tilde{\vv w}_1\right\|\le \|\vv a_1\| \le \red{\sqrt i}\rho$. Also, observe that
$\tilde{\vv w}_1\wedge \tilde{\vv w}^{(i-1)}_2 = a_{1,2} \tilde{\vv w}^{(i)}$. Then, using
\eqref{eq:ly13} and the earlier made assumption $|a_{1,2}|> \rho L^{-\frac{1}{2}}$, we get that
\begin{align}
\left\|\tilde{\vv w}^{(i)}  \right\|
&=\frac{1}{|a_{1,2}|} \left\|\tilde{\vv w}_1\wedge \tilde{\vv w}^{(i-1)}_2\right\| ~\stackrel{\eqref{submult}}{\le}~ \frac{1}{|a_{1,2}|} \left\|\tilde{\vv w}_1\right\| \left\| \tilde{\vv w}^{(i-1)}_2\right\|\nonumber\\
&\stackrel{\left\|\tilde{\vv w}_1\right\|\le \red{\sqrt i}\rho}{\le}~ \frac{\red{\sqrt i}\rho}{|a_{1,2}|}\left\| \tilde{\vv w}^{(i-1)}_2\right\|
~\stackrel{|a_{1,2}|>\rho L^{-\frac{1}{2}}}{\le}~
\red{\sqrt i}L^{\frac{1}{2}} \left\| \tilde{\vv w}^{(i-1)}_2\right\| \nonumber\\
&\stackrel{\eqref{eq:ly13}}{\le}
\red{\sqrt i}L^{\frac{1}{2}} \cdot 2  \rho^i L^{-1} = 2 \red{\sqrt i}\rho^i L^{-\frac{1}{2}}\,.\label{vb2876}
\end{align}
Then, using \eqref{eq:ly13}, \eqref{vb2876}, the definition of $\vv w^{(i)}$ and the condition $L\ge\red{1}$ we get that
$$
\|\vv w^{(i)}\| \le \|\vv e_2 \wedge \tilde{\vv w}_2^{(i-1)}\| + \|\tilde{\vv w}^{(i)}\|\le 2\rho^{i} L^{-1}
+2\red{\sqrt i}\rho^{i} L^{-\frac{1}{2}}\le \red{4\sqrt n}\rho^{i} L^{-\frac{1}{2}}\,.
$$
This verifies the right hand side of \eqref{vb641} and completes the proof of Lemma~\ref{lm:yang-lm5.8}.
\end{proof}

Now we are fully equipped to establish Proposition \ref{prop:estimate-dangerous}.

\begin{proof}[Proof of Proposition \ref{prop:estimate-dangerous}]
Recall from \S\ref{sec2.4} that we only need to prove it for $q > \etainvt \mo$ and $p = q-4l$, where $ \m \le l < \etaqt $
since otherwise $\hat{\cJ}_{p,q} = \emptyset$. With this in mind, fix any interval $\Ip \in \cJ_p$. Let $N_l$ denote
the number of intervals $\Iq \in \cI_{q+1}$ such that $\Iq \subset \Ip$, $\Iq \not\in \hat{\cJ}_{p', q}$ for $p' =q$ and any $0 \le p' < p$ and such that \eqref{pr2} holds. By \eqref{lengthIq+1} and \eqref{Jqq}, we have that \eqref{LB2} holds. In turn, by \eqref{pr2} and Lemma~\ref{extension}, we have that for any $\justx\in \Iq$ \eqref{vb590} holds for any of these intervals $\Iq$.
Suppose that for all $1 \le i \le n$ and $\vv v = \vv v_1 \wedge \cdots \wedge \vv v_i \in \bigwedge^i \left(\Z^{n+1}\right) \setminus \{ \vv 0\}$ we have that
\begin{equation}\label{vb972}
\max_{\justx \in \Ip} \|H(\justx) \vv v\|   \ge 1\,,
\end{equation}
where $H$ is given by \eqref{H}. Then, letting $\rho=1$, $g_{\vv \tau} = b\left(\beta' l\right)  a(\beta(q+1))$, by Theorem~\ref{thm:QnD-fractal-local} and \eqref{Ahlfors2}, we obtain that
\begin{equation}\label{vb3:28}
\mu \left(\left\{ \justx \in \Ip : \lat{H(\justx)} \not\in K_{3 e^{ - \epsilon\beta l}} \right\} \right)
\le CM_1 3^{\alpha+\gamma} e^{-\gamma\epsilon\beta l}|I_0|^\alpha R^{-p\alpha}\,.
\end{equation}
On the other hand, by \eqref{LB2} and Lemma~\ref{extension}, we have that this measure is $\ge N_l\times (3C)^{-1}|I_0|^\alpha R^{-\alpha(q+1)}$. Hence, using \eqref{betas} and $q-p=4l$ gives
\begin{align}
N_l&\le CM_1 3^{\alpha+\gamma} e^{ - \gamma\epsilon\beta l}|I_0|^\alpha R^{-p\alpha}\cdot 3C|I_0|^{-\alpha} R^{\alpha(q+1)}\nonumber\\
& =C^2M_13^{1+\alpha+\gamma} R^{\alpha(q+1-p) - \frac{\gamma\epsilon}{4(1+r_1)} (q-p)}\,.\label{vb834}
\end{align}
Note that since $p<q$ we have that $\frac12(q+1-p)\le (q-p)$. Then \eqref{vb834} implies the desired estimate with $\eta_2 = \frac{\gamma \epsilon}{8\alpha (1+r_1)}$ and $C_2=C^2M_1 3^{1+\alpha+\gamma}$. Thus, to complete the proof of this proposition, it is enough to demonstrate that \eqref{vb972} always takes place for the intervals $\Ip$ and the multi-vectors $\vv v$ in question. We will show this under the assumption that $N_l\ge1$ since, if $N_l=0$, then the required bound holds anyway.

Suppose, for a contradiction, that there exists $\vv v = \vv v_1 \wedge \cdots \wedge \vv v_i \in \bigwedge^i \left(\Z^{n+1}\right) \setminus \{ \vv 0\}$ with $1 \le i \le n$ such that
\begin{equation}\label{vb094}
\max_{\justx\in \Ip}  \|H(\justx) \vv v\|  < 1\,.
\end{equation}
Since $N_l\ge1$, there exists an interval of the Cantor construction $\Iq \subset \Ip$ such that $\Iq \in\hat{\cJ}_{p,q}$. Take any $\xzero\in \Iq$. Then the point $\justx=\xzero+\theta R^{-q-1+l'}$ belongs to $\Ip$ for either all $\theta\in\left[0,\tfrac12|I_0|\right]$ or all $\theta\in\left[-\tfrac12|I_0|,0\right]$ and $l'=4l+1$. Without loss of generality we will assume that this holds for any $\theta\in\left[0,\tfrac12|I_0|\right]$. Then, by \eqref{vb094} and Lemma~\ref{extension2}, namely the right hand side of \eqref{vb590A}, we get that
\begin{equation}\label{eq:ly6}
\textstyle\left\|u\left(\theta R^{3l+1} \vv e_1\right) H(\xzero)\vv v\right\|  \le 2\qquad\text{for all $\theta\in \left[0,\frac{1}{2}|I_0|\right]$}\,,
\end{equation}
when $R$ is sufficiently large. Let $\vv a_j =  H(\xzero) \vv v_j$ for $j=1,\dots,i$, $\rho=2^{\frac12}$ and $L=\frac12|I_0|R^{3l+1}$. Assuming that $R\ge |I_0|^{-1}$, we have that $L\ge1$. Then, in view of \eqref{eq:ly6}, Lemma~\ref{lm:yang-lm5.8} is applicable. Thus, by Lemma~\ref{lm:yang-lm5.8}, for the remainder of the proof we can assume the validity of either case \caseA{} or case \caseB{} of the lemma.

\underline{Case \caseA{}}. Let $\vv a$ be as in Case \caseA{} of Lemma~\ref{lm:yang-lm5.8}. Then
\begin{equation}\label{vva}
  \vv a=H(\xzero)\vv v\qquad\text{for some }\vv v\in\Z^{n+1}\setminus\{\vv0\}\,,
\end{equation}
\begin{equation}\label{vb2858}
|a_2|\le \rho L^{-\frac{1}{2}}\le2\left(2|I_0|^{-1}R^{-3l-1}\right)^{\frac12} \le R^{-3l/2}=e^{-\frac32(1+1/n)\beta'l}\,,
\end{equation}
provided that $R\geq8|I_0|^{-1}$. Furthermore, $\|\vv a\|\le\rho\le 2$.
Then, using \eqref{eq:b(t)}, \eqref{betas}, \eqref{vb2858} and $\|\vv a\|\le2$, we get that
\begin{equation}\label{vb8768}
\left\|b\left(\beta' l\right) \vv a\right\| \le e^{-\beta' l /n}\|\vv a\|+e^{\beta' l}|a_2|  \le 3 e^{-3\epsilon \beta l}
< e^{-2\epsilon \beta l}
\end{equation}
provided that $R\geq 3^{\frac{1+r_1}{\epsilon}}$.
In \eqref{vb8768} we used the facts that $\beta' l /n > 3 \epsilon \beta l$ and $e^\beta=R^{\frac{1}{1+r_1}}$ which follow from \eqref{betas} and \eqref{eq:eps}.
Then, by \eqref{vva} and \eqref{vb8768}, we get that
\begin{equation}\label{eq:ly7}
\lat{b\left(\beta'l\right)H(\xzero)}\stackrel{\eqref{H}}{=}\lat{b\left(\beta' (2l)\right) a(\beta(q+1)) z(\xzero) u(\varphi(\xzero))} \not\in K_{e^{-\epsilon \beta (2l)}}\,.
\end{equation}
Recall that $1\le l < \etaqt$. Then, by \eqref{eq:ly7}, we have that $\Iq\in\hat{\cJ}_{p', q}$ with $p'=q-8l<p$ if $l < \etaqf$ and with $p'=0$ if $l \ge \etaqf$. In view of the definition of $\hat{\cJ}_{p,q}$ (see \S\ref{sec2.4}), this leads to a contradiction. The proof in Case \caseA{} is thus completed.


\underline{Case \caseB{}}.
In this case there exist $\vv w^{(i-1)} \in \bigwedge\nolimits^{i-1} W$ and $\vv w^{(i)} \in \bigwedge^i W$, where $W$ is the same as in Lemma~\ref{lm:yang-lm5.8}, satisfying \eqref{vb640}, $\|\vv w^{(i-1)}\| \le \rho^i=2$ and
\begin{equation}\label{vb641+}
\|\vv w^{(i)}\| \le \red{8\sqrt n} \cdot 2^{\frac{1}{2}}|I_0|^{-\frac{1}{2}}R^{-(3l+1)/2}\le \red{12\sqrt n\,}|I_0|^{-\frac{1}{2}}R^{-(3l+1)/2}\,.
\end{equation}
Next, using \eqref{weights}, \eqref{betas}, \eqref{vb641+} and basic properties of actions on linear maps on multivectors we calculate that
\[
\left\|a\left(-\beta l\right) \left(e_1 \wedge \vv w^{(i-1)}\right)\right\| \le e^{- r_n\beta l} \left\|e_1 \wedge \vv w^{(i-1)}\right\| = e^{- r_n\beta l} \left\|\vv w^{(i-1)}\right\| \leq  2R^{-\frac{r_n}{1+r_1}l}\,,
\]
\begin{align*}
\left\|a\left(-\beta l\right) \vv w^{(i)}\right\|  \le  e^{\beta l}\left\|\vv w^{(i)}\right\| &\leq \red{12\sqrt n\,}|I_0|^{-\frac12}R^{-\left(\frac{3}{2}l+\frac{1}{2}\right)} R^{\frac{l}{1+r_1}}\\
&=
\red{12\sqrt n\,}|I_0|^{-\frac12}R^{-\frac{1}{2}}R^{-\left(\frac{3}{2}-\frac{1}{1+r_1}\right)l}\,.
\end{align*}
Note that $0<r_1<1$. Then combining the above two estimates with \eqref{vb640} gives
$$
\left\|a\left(-\beta l\right) (\vv a_1 \wedge \cdots \wedge \vv a_i)\right\| \le \red{n^{-\frac{n}{2}}}R^{- \frac{r_n l}{2}}
$$
provided that $R$ is sufficiently large.
Therefore, by Minkowski's Theorem, there exists
$$
\vv c\in \Span_\Z\left( a\left(-\beta l\right) \vv a_1, \dots, a\left(-\beta l\right)\vv a_i\right)\setminus\{\vv0\}
$$
such that
$\|\vv c\| \le R^{- \frac{r_n l}{2i}}$. Therefore, by \eqref{betas} and \eqref{eq:eps}, we have that
\begin{equation}\label{vvc}
\|\vv c\| \le e^{-\epsilon \beta l}.
\end{equation}
In view of the definition of the vectors $\vv a_j$ and the choice of $\vv c$, we have that $\vv c=a(-\beta l)H(\xzero) \vv v$ for some $\vv v\in\Z^{n+1}\setminus\{\vv0\}$. This together with \eqref{H}, \eqref{vvc} and the trivial fact that $a\left(-\beta l\right)$ and $b(\beta' l)$ commute implies that
\begin{equation}\label{eq:ly7+}
\lat{b\left(\beta' l\right) a(\beta(q+1 -l)) z(x_0) u(\varphi(x_0))} \not\in K_{e^{-\epsilon \beta l}}\,.
\end{equation}
Recall that $1\le l < \etaqt$ and so $l < (q-l)/4$. Then, by \eqref{eq:ly7+} and the fact that $\xzero\in I$, we have that $\Iq$ is contained in an interval from $\hat{\cJ}_{p', q-l}$ with $p'=(q-l) - 4l<p=q-4l$ if $l < (q-l)/8$ and with $p'=0$ otherwise. However, this is impossible since, by the construction of the Cantor set (see \S\ref{sec2.4}) any such interval would have been removed earlier than at step $q$ of the Cantor set construction. This completes the proof in Case \caseB{} and thus completes the proof of this proposition.
\end{proof}

\section{Proof of the main theorem}\label{theproof}

By Propositions \ref{prop:estimate-extremely-dangerous} and \ref{prop:estimate-dangerous},
if $R$ is sufficiently large then there exist positive constants $R_3$, $C_3$ and $\eta_3$ such that whenever $R\ge R_3$ we have that for all $0 \le p < q$, \eqref{hpqbound} holds with
\begin{equation}\label{hpqbound0}
\h_{p,q} \le C_3 R^{\alpha(1-\eta_3) (q+1 -p)}.
\end{equation}
Note that, by Proposition \ref{Prop_p=q}, \eqref{hpqbound} holds for $p=q$ with $\h_{q,q} \le R - (4C)^{-2} R^{\alpha}$.
Let $(t_q)_{q \in \N}$ be defined as in Theorem~\ref{thm:generalized-cantor-set-estimate}.
We shall prove that
\begin{equation}\label{claim0}
t_q \ge (6C)^{-2} R^{\alpha}
\end{equation}
for all $q \in \N$. Recall that here $C$ and $\alpha$ are the parameters characterising $\mu$, that is $\mu$ is $(C,\alpha)$-Ahlfors regular.

We shall prove \eqref{claim0} by induction. To begin with, note that $t_0 = R - \h_{0,0} \ge (4C)^{-2} R^{\alpha}$ and so \eqref{claim0} holds for $q=0$. Now suppose that $q_1>0$  and \eqref{claim0} holds for every $q \le q_1 -1$. By \eqref{t_q}, \eqref{hqq}, \eqref{hpqbound0} and \eqref{claim0} for $q=1,\ldots,q_1-1$, we have that
\begin{align}
t_{q_1} &= R - \h_{q_1,q_1} - \sum_{j=1}^{q_1} \frac{\h_{q_1-j, q_1}}{\prod_{i=1}^j t_{q_1-i}}\nonumber\\
 & \ge (4C)^{-2} R^{\alpha} - \sum_{j=1}^{q_1} \frac{ C_3 R^{\alpha(1-\eta_3)(j+1)}}{((6C)^{-2} R^{\alpha})^j} \nonumber\\
          & \ge (4C)^{-2} R^{\alpha} - R^{\alpha} \left( C_3 R^{-\eta_3 \alpha} \sum_{j=1}^{\infty} \left(\frac{(6C)^2}{R^{\eta_3 \alpha}}\right)^j \right)\,.\label{vb9782}
\end{align}
By choosing $R$ large enough, we can ensure that $C_3 R^{-\eta_3 \alpha} \le  \tfrac12(4C)^{-2}$ and
$$
\sum_{j=1}^{\infty} \left(\frac{(6C)^2}{R^{\eta_3 \alpha}}\right)^j \le 1\,.
$$
Then, by \eqref{vb9782}, we get that
$
t_{q_1} \ge (4C)^{-2} R^{\alpha} -  \tfrac12(4C)^{-2} R^{\alpha} \ge (6C)^{-2} R^{\alpha}.
$
This verifies \eqref{claim0} for $q=q_1$ and completes the induction step. By Theorem~\ref{thm:generalized-cantor-set-estimate},
we have that $\cK_{\infty} \neq \emptyset$. Hence, by Proposition~\ref{prop2.3},
we have \eqref{equ:intersection-fractal}. This completes the proof of Theorem~\ref{thm:main-theorem2},
and, by \eqref{equiv}, completes the proof of Theorem~\ref{thm:main-theorem}.

\section{Quantitative non-divergence estimates}\label{sec-QnD}

Throughout this section, $B(x,r)$ will denote a ball in $\R^d$ of radius $r$ centred at $x$, $K_\varepsilon$ is defined as in \eqref{K_eps} and $\mu$ is a locally finite Borel measure on $\R^d$. Given a ball $B=B(x,r)$ and $\lambda>0$, $\lambda B$ will denote the ball $B(x,\lambda r)$. The primary goal of this section is to prove Theorems~\ref{thm:QnD-fractal-local} and \ref{thm:QnD-fractal-global}. Furthermore, we will obtain more general quantitative non-divergence estimates, which are of independent interest. We start by recalling a general result established in \cite{MR2134453}.

\subsection{A general quantitative non-divergence estimate}

To begin with, we recall some definitions from \cite{MR2134453}.

\begin{defn}[See {\cite[\S2]{MR2134453}}]\label{def_Fed_fun}
Given an open subset $U\subset\R^d$, a measure $\mu$ is called {\em $D$-Federer on $U$} if for any $x\in\supp\mu\cap U$ and any $r>0$ such that $B(x,3r)\subset U$ one has that
$$
\mu(B(x,3r))<D\mu(B(x,r))\,.
$$
\end{defn}

\begin{defn}[See {\cite[\S4]{MR2134453}}]\label{def_good_fun}
Given an open subset $U\subset\R^d$, a $\mu$-measurable function $f:U\to\R$ is called {\em $(C,\alpha)$-good with respect to $(${\em abbr.} w.r.t.$)$ $\mu$} if for any non-empty open ball $B\subset U$ centred in $\supp\mu$ one has that
$$
\forall~\varepsilon>0 \qquad \mu\left(\{x\in B:|f(x)|<\varepsilon\}\right)\le C\left(\frac{\varepsilon}{\|f\|_{\mu,B}}\right)^\alpha\mu(B)\,,
$$
where
$$
\|f\|_{\mu,B}=\sup\left\{|f(x)|:x\in\supp\mu\cap B\right\}\,.
$$
\end{defn}

\begin{remark}
Note that in Definition~\ref{def_good_fun} and in other definitions of this section the constants $C$ and $\alpha$ are not necessarily the same as those used in \S\ref{sec-intersections-fractals} with the $(C,\alpha)$-Ahlfors regular measure $\mu$.
\end{remark}

The following theorem is a slightly simplified version of Theorem~4.3 from \cite{MR2134453}, which is a generalisation of the quantitative non-divergence estimate of Kleinbock and Margulis \cite{Klein_Mar} for Lebesgue measure.

\begin{thm}\label{KleinbockQnD}
Let $n,d\in\N$ and $C,D,\alpha>0$. Then there exists a positive constant $C'$ with the following property. Suppose that $0<\rho\le1$, $\mu$ is a locally finite Borel measure and $B$ is a non-empty open ball in $\R^d$ centred in $\supp\mu$. Suppose $\mu$ is $D$-Federer on $\tilde B=3^{n+1}B$ and suppose that $\HH:\tilde B\to \SL{n+1}{\R}$ is a continuous map such that for any primitive collection $\vv v_1,\dots,\vv v_r\in\Z^{n+1}$
\begin{enumerate}
  \item[{\rm(i)}]  the function $x\mapsto\|\HH(x)\vv v_1\wedge\cdots\wedge \HH(x)\vv v_r\|$ is $(C,\alpha)$-good on $\tilde B$ w.r.t. $\mu$; and

      \medskip

  \item[{\rm(ii)}] $\sup\limits_{x\in\supp\mu\cap B}\|\HH(x)\vv v_1\wedge\cdots\wedge \HH(x)\vv v_r\|\ge \rho$.
\end{enumerate}
Then for any $\varepsilon>0$
\begin{equation*}
\mu\left(\left\{x\in B:\HH(x)\Z^{n+1}\not\in K_{\epsilon} \right\}\right)\le C'\left(\frac{\varepsilon}{\rho}\right)^\alpha\mu(B)\,.
\end{equation*}
\end{thm}

\begin{remark}
Note that $C'$ can be taken to be $(n+1)C\left(N_d D^2\right)^{n+1}$, where $N_d$ is the Besicovitch constant, see \cite[Theorem~2.2]{MR2434296}.
\end{remark}

Now we state and prove a version of Theorem~\ref{KleinbockQnD} in which condition (ii) is relaxed at the expense of a more restrictive version of condition (i) utilising the notion of absolutely good functions which is now recalled.

\begin{defn}[See {\cite[\S7]{MR2134453}}]\label{def_abs_good_fun}
Given an open subset $U\subset\R^d$, a $\mu$-measurable function $f:U\to\R$ is called {\em absolutely $(C,\alpha)$-good on $U$ w.r.t $\mu$} if for any non-empty open ball $B\subset U$ centred in $\supp\mu$ one has that
\begin{equation*}
\forall~\varepsilon>0 \qquad \mu(\{x\in B:|f(x)|<\varepsilon\})\le C\left(\frac{\varepsilon}{\|f\|_B}\right)^\alpha\mu(B)\,,
\end{equation*}
where $\|f\|_B:=\sup_{x\in B}|f(x)|$. We say that $f$ is {\em (absolutely) good on $U$ w.r.t. $\mu$} if $f$ is (absolutely) $(C,\alpha)$-good on $U$ w.r.t. $\mu$ for some $C>0$ and $\alpha>0$.
\end{defn}

Note that since obviously $\|f\|_{\mu,B}\le \|f\|_{B}$ we trivially have the following lemma.

\begin{lemma}\label{lem5.13}
Any function $f:U\to\R$ defined on an open set $U\subset\R^d$, which is absolutely $(C,\alpha)$-good on $U$ w.r.t a measure $\mu$, is $(C,\alpha)$-good on $U$ w.r.t $\mu$.
\end{lemma}

\begin{thm}\label{KleinbockQnD2}
Let $n,d\in\N$ and $C,D,\alpha>0$. Then there exists a positive constant $C''=C''(n,d,C,D,\alpha)$ with the following property. Suppose that $0<\rho\le1$, $\mu$ is a locally finite Borel measure and $B$ is a non-empty open ball in $\R^d$ centred in $\supp\mu$. Suppose $\mu$ is $D$-Federer on $\tilde B=3^{n+1}B$ and suppose that $\HH:\tilde B\to \SL{n+1}{\R}$ is a continuous map such that for any primitive collection $\vv v_1,\dots,\vv v_r\in\Z^{n+1}$
\begin{enumerate}
  \item[{\rm(i*)}]  the function $x\mapsto\|\HH(x)\vv v_1\wedge\cdots\wedge \HH(x)\vv v_r\|$ is absolutely $(C,\alpha)$-good on $\tilde B$ w.r.t. $\mu$; and

      \smallskip

  \item[{\rm(ii*)}] $\sup\limits_{x\in B}\|\HH(x)\vv v_1\wedge\cdots\wedge \HH(x)\vv v_r\|\ge \rho$.
\end{enumerate}
Then for any $\varepsilon>0$
\begin{equation}\label{QnD2}
\mu\left(\left\{x\in B:\HH(x)\Z^{n+1}\not\in K_{\epsilon} \right\}\right)\le C''\left(\frac{\varepsilon}{\rho}\right)^\alpha\mu(B)\,.
\end{equation}
\end{thm}

\begin{proof}
First of all, by Lemma~\ref{lem5.13}, condition (i*) verifies condition (i) of Theorem~\ref{KleinbockQnD}. Now we verify condition (ii) of Theorem~\ref{KleinbockQnD}. Let $\vv v_1,\dots,\vv v_r\in\Z^{n+1}$ be a primitive collection and let $f(x)=\|\HH(x)\vv v_1\wedge\cdots\wedge \HH(x)\vv v_r\|$. Then, by (ii*),
$\|f\|_{B}\ge\rho$.
Let $\rho'$ be such that $\max\{1,C\}\cdot(\rho'/\rho)^\alpha=\tfrac12$.  Since $0<\rho\le1$, we have that $0<\rho'\le 1$. By (i*), $f$ is absolutely $(C,\alpha)$-good on $B$ w.r.t. $\mu$. Therefore, since $B$ is centred in $\supp\mu$, by Definition~\ref{def_abs_good_fun}, we have that
$$
\mu\left(\{x\in B:|f(x)|<\rho'\}\right)\le C\left(\frac{\rho'}{\|f\|_{B}}\right)^{\alpha}\mu(B)\le
C\left(\frac{\rho'}{\rho}\right)^{\alpha}\mu(B)\le \tfrac12\mu(B)<\mu(B)\,.
$$
Therefore, there exists $x\in \supp\mu\cap B$ such that $|f(x)|\ge\rho'$. This verifies condition (ii) in Theorem~\ref{KleinbockQnD} with $\rho$ replaced by $\rho'$. Hence, by Theorem~\ref{KleinbockQnD}, we get \eqref{QnD2} with $C''=2C'\max\{1,C\}$ and the proof is complete.
\end{proof}

\subsection{Quantitative non-divergence estimate for analytic maps}
Now the goal is to specialise Theorem~\ref{KleinbockQnD2} to the case of analytic maps. We begin by recalling the notion of decaying measures. In what follows
$d_{\cH}(x)=\inf\{\|x-x'\|:x'\in\cH\}$ is the Euclidean distance of $x\in\R^d$ from $\cH\subset\R^d$ and
$B(\cH,\varepsilon)=\{x\in\R^d:d_{\cH}(x)<\varepsilon\}$ is the $\varepsilon$-neighborhood of $\cH$.

\begin{defn}[See {\cite[\S2]{MR2134453}}]\label{def_abs_good_mes}
Given an open subset $U\subset\R^d$, a measure $\mu$ on $\R^d$ is called {\em $(C,\alpha)$-decaying on $U$} if for any open ball $B\subset U$ of radius $r_B>0$ centred in $\supp\mu$ and any hyperplane $\cH\subset\R^d$ one has that
$$
\forall~\varepsilon>0 \qquad \mu(B\cap B(\cH,\varepsilon))\le C\left(\frac{\varepsilon}{\|d_{\cH}\|_{\mu,B}}\right)^\alpha\mu(B)\,,
$$
where  $\|d_{\cH}\|_{\mu,B}=\sup\{d_{\cH}(x):x\in\supp\mu\cap B\}$.
The measure $\mu$ is called {\em absolutely $(C,\alpha)$-decaying on $U$} if for any non-empty open ball $B\subset U$ of radius $r_B>0$ centred in $\supp\mu$ and any hyperplane $\cH\subset\R^d$ one has that
$$
\forall~\varepsilon>0 \qquad \mu(B\cap B(\cH,\varepsilon))\le C\left(\frac{\varepsilon}{r_B}\right)^\alpha\mu(B)\,.
$$
We say that $\mu$ is {\em (absolutely) decaying on $U$} if $\mu$ is (absolutely) $(C,\alpha)$-decaying on $U$ for some $C>0$ and $\alpha>0$.
\end{defn}

Note that $\|d_{\cH}\|_{\mu,B} \le 2r_B$ and therefore we have the following lemma.

\begin{lemma}\label{lem5.11}
Any measure $\mu$, which is absolutely $(C,\alpha)$-decaying on an open set $U\subset\R^d$, is $(2^\alpha C,\alpha)$-decaying on $U$.
\end{lemma}

\begin{thm}\label{KleinbockQnD3}
Let $n,d\in\N$, $D>0$, $U\subset\R^d$ be open, $\HH_0:U\to \SL{n+1}{\R}$ be a map such that every entry of $\HH_0$ is a real analytic function and let $\mu$ be a locally finite Borel measure, which is absolutely decaying and $D$-Federer on $U$. Then for any $x_0\in \supp\mu\cap U$ there exists a ball $B(x_0)\subset U$ centred at $x_0$ and constants $C_0,\alpha_0>0$ such that for any ball $B\subset B(x_0)$, any diagonal matrix $g\in\SL{n+1}{\R}$ and any $0<\rho\le1$ at least one of the following two conclusions holds for $\HH(x)=g\HH_0(x)$\\[-2ex]
\begin{enumerate}
\item[$(1)$] for all $\varepsilon>0$
\begin{equation}\label{QnD3}
\mu\left(\left\{x\in B:\HH(x)\Z^{n+1}\not\in K_{\epsilon} \right\}\right)\le C_0\left(\frac{\varepsilon}{\rho}\right)^{\alpha_0}\mu(3B)\,.
\end{equation}
\item[$(2)$] there exists a primitive collection $\vv v_1,\dots,\vv v_r\in\Z^{n+1}$ with $1 \le r \le n$ such that
$$
\sup\limits_{x\in B}\|\HH(x)\vv v_1\wedge\cdots\wedge \HH(x)\vv v_r\|< \rho\,.
$$
\end{enumerate}
\end{thm}

The proof will require several auxiliary results. We start with the following two lemmas that easily follow from the definitions of good functions.

\begin{lemma}\label{Calpha2}
Let $U\subset\R^d$ be open, $C,\alpha>0$ and $f:U\to\R$ be a $\mu$-measurable function.  If $f$ is (absolutely) $(C,\alpha)$-good on $U$ w.r.t. $\mu$, then $\lambda f$ is (absolutely)
$(C',\alpha')$-good on $U'$ w.r.t. $\mu$ for every $C' \geq \max\{C,1\}$, $0<\alpha' \leq \alpha$, any open subset $U'\subset U$ and any $\lambda\in\R$. Furthermore, $f$ is (absolutely) $(C,\alpha)$-good on $U$ w.r.t. $\mu$ if and only if so is $|f|$.
\end{lemma}

\begin{proof}
In the case of $(C,\alpha)$-good functions the proof of this lemma can be found in \cite[Lemma~2.1]{MR2314053} and \cite[Lemma~3.1]{Bernik-Kleinbock-Margulis}. The proof in the case of absolutely $(C,\alpha)$-good functions is nearly identical with the only change is that the supremum norm $\|\cdot\|_B$ of functions is replaced by their $\mu$-essential supremum $\|\cdot\|_{\mu,B}$. We leave further details to the reader.
\end{proof}

\begin{lemma}\label{Calpha1}
Suppose that the functions $f_1,\dots,f_N$ are (absolutely) $(C,\alpha)$-good on an open subset $U\subset\R^d$ w.r.t. $\mu$, then $f=\left(f_1^2+\dots+f_N^2\right)^{1/2}$ is (absolutely) $\left(N^{\alpha/2}C,\alpha\right)$-good on $U$ w.r.t. $\mu$.
\end{lemma}

\begin{proof}
For $(C,\alpha)$-good functions this can be found in \cite[Lemma~4.1]{MR2134453}. In the case of absolutely $(C,\alpha)$-good functions the proof is adapted as in Lemma~\ref{Calpha2}.
\end{proof}

The following statement uses the definition of non-degeneracy of maps introduced in \cite{Klein_Mar}, which is now recalled.
\begin{defn}[See {\cite[\S2]{MR2134453}}]\label{nondegen}
Given an open subset $U\subset\R^d$ and $l\in\N$, a map $\vv f=(f_1,\dots,f_N) : U \to \R^N$ is called {\em $\ell$-nondegenerate at $x_0\in U$} if there is an open neighborhood $V\subset U$ of $x_0$ such that $\vv f$ is $C^\ell$ on $V$ and
\begin{equation}\label{span2}
\Span_\R\left\{\frac{\partial^{k_1+\dots+k_d}\vv f}{\partial x_1^{k_1}\dots\partial x_d^{k_d}}(x_0) :1\le k_1+\dots+k_d\le l\right\}=\R^N\,.
\end{equation}
\end{defn}

\begin{prop}[Proposition 7.3 in \cite{MR2134453}]\label{prop614}
Let $U\subset\R^d$ be open and let $\vv f:U\to\R^N$ be a $C^{\ell+1}$ map which is $\ell$-nondegenerate at $x_0\in U$. Let
$\mu$ be a measure which is $D$-Federer and absolutely $(C,\alpha)$-decaying on $U$ for some $C,D,\alpha>0$. Then, there exists an open neighborhood $V\subset U$ of $x_0$ and constant $\tilde C>0$ such that any
$$
f\in \cS_{\vv f}:=\left\{c_0+c_1f_1+\dots+c_Nf_N:\vv c=(c_0,\dots,c_N)\in\R^{N+1}, ~\|\vv c\|=1\right\}
$$
is absolutely $\left(\tilde C,\frac{\alpha}{2^{\ell+1}-2}\right)$-good on $V$ w.r.t. $\mu$.
\end{prop}

\begin{cor}\label{prop614+}
Let $U\subset\R^d$ be open and let $g_1,\dots,g_M:U\to\R$ be real analytic functions. Let
$\mu$ be a measure which is absolutely decaying and $D$-Federer on $U$ for some $D>0$. Then, for every $x_0\in U$ there exists an open neighborhood $V\subset U$ of $x_0$ and constants $\tilde C,\alpha_0>0$ such that any linear combination of
$g_1,\dots,g_M$ with real coefficients is absolutely $\left(\tilde C,\alpha_0\right)$-good on $V$ w.r.t. $\mu$.
\end{cor}

\begin{proof}
Fix any $x_0\in U$. Since $g_1,\dots,g_M$ are analytic, there exists a neighborhood $U'\subset U$ of $x_0$ such that each of these functions can be expanded into a Taylor series at $x_0$ for all $x\in U'$. Let $f_1,\dots,f_N$ be any maximal sub-collection of $g_1,\dots,g_M$ such that
\begin{equation}\label{vb27}
1,f_1,\dots,f_N
\end{equation}
are linearly independent over $\R$ as functions of $x\in U'$. Then the map $\vv f=(f_1,\dots,f_N)$ is $\ell$-non-degenerate at $x_0$ for some $\ell\in\N$. Indeed, if this was not the case, then the span on the left of \eqref{span2} would be a proper linear subspace, say $L$, of $\R^N$. By the Taylor expansion of $\vv f$, we would then have that $\vv f(x)\in \vv f(x_0)+L$ for all $x\in U'$, contrary to the linear independence of \eqref{vb27}.
Since $\vv f$ is $l$-non-degenerate at $x_0$ and $\mu$ is absolutely decaying, by Proposition~\ref{prop614}, there exists a neighborhood $V\subset U'$ of $x_0$ and some positive constants $\tilde C$ and $\alpha_0$ such that every function in $\cS_{\vv f}$ is absolutely $\left(\tilde C,\alpha_0\right)$-good on $V$ w.r.t. $\mu$.
By the maximality of the sub-collection \eqref{vb27}, any linear combination of
$g_1,\dots,g_M$ with real coefficients is a constant multiple of an element from $\cS_{\vv f}$. Therefore, by Lemma~\ref{Calpha2}, any such linear combination is also absolutely $\left(\tilde C,\alpha_0\right)$-good on $V$ w.r.t. $\mu$. The proof is thus complete.
\end{proof}

\begin{proof}[Proof of Theorem~\ref{KleinbockQnD3}]
Let $g_1,\dots,g_M:U\to\R$ be the collection of all minors of $\HH_0$. Since $\HH_0$ is analytic, $g_1,\dots,g_M$ are analytic. Take any $x_0\in \supp\mu\cap U$. Then, by Corollary~\ref{prop614+}, there exists an open neighborhood $V\subset U$ of $x_0$ and constants $\tilde C,\alpha_0>0$ such that any linear combination of $g_1,\dots,g_M$ over $\R$ is absolutely $\left(\tilde C,\alpha_0\right)$-good on $V$ w.r.t. $\mu$. Let $B(x_0)$ be a ball centred at $x_0$ such that $3^{n+2}B(x_0)\subset V$ and let $C_0=C''(n,d,C,D,\alpha_0)$, where $C''$ is as in Theorem~\ref{KleinbockQnD2} and $C=2^{(n+1)\alpha_0/2}\tilde C$. We claim that $B(x_0)$, $C_0$ and $\alpha_0$ satisfy Theorem~\ref{KleinbockQnD3}.

To prove this claim take any ball $B\subset B(x_0)$ and any $0<\rho\le1$. If $\mu(B)=0$ then \eqref{QnD3} is trivially true and we are done. Otherwise, there exists $x'\in\supp\mu\cap B$. Let $B'$ be the ball of radius twice the radius of $B$ centred at $x'$. Then, $B\subset B'\subset 3B$ and also $3^{n+1}B'\subset 3^{n+2}B\subset 3^{n+2}B(x_0)\subset V$.

Take a primitive collection $\vv v_1,\dots,\vv v_r\in\Z^{n+1}$.
Then, since $\HH(x)=g\HH_0(x)$ and $g$ is constant, the coordinate functions of $\HH(x)\vv v_1\wedge\cdots\wedge \HH(x)\vv v_r$ are linear combinations of $g_1,\dots,g_M$ and thus, by what we have shown above, they are absolutely $\left(\tilde C,\alpha_0\right)$-good on $\tilde B'=3^{n+1}B'$ w.r.t. $\mu$. By Lemma~\ref{Calpha1}, $x\mapsto \|\HH(x)\vv v_1\wedge\cdots\wedge \HH(x)\vv v_r\|$ is absolutely $\left(2^{(n+1)\alpha_0/2}\tilde C,\alpha_0\right)$-good on $\tilde B'$ w.r.t. $\mu$. This verifies condition (i*) of Theorem~\ref{KleinbockQnD2} with $B'$ in place of $B$.

Further, we can assume that $\sup_{x\in B'}\|\HH(x)\vv v_1\wedge\cdots\wedge \HH(x)\vv v_r\|\ge \rho$ as otherwise, since $B\subset B'$, conclusion (2) of Theorem~\ref{KleinbockQnD3} holds and we are done. This verifies condition (ii*) of Theorem~\ref{KleinbockQnD2} with $B'$ in place of $B$.
Hence, by Theorem~\ref{KleinbockQnD2}, we get that for any $\varepsilon>0$
\begin{equation}\label{QnD2+}
\mu\left(\left\{x\in B':\HH(x)\Z^{n+1}\not\in K_{\epsilon} \right\}\right)\le C''\left(\frac{\varepsilon}{\rho}\right)^{\alpha_0}\mu(B')\,.
\end{equation}
Since $B\subset B'\subset 3B$, \eqref{QnD2+} implies \eqref{QnD3} and completes the proof.
\end{proof}

\subsection{Proof of the local estimate}
Now the goal is to prove Theorem~\ref{thm:QnD-fractal-local}.
We will need the following two well known statements, which we prove for completeness.

\begin{lemma}\label{lem613}
Suppose that $\mu$ is a $(C,\alpha)$-Ahlfors regular measure on $\R^d$.
Then $\mu$ is $D$-Federer for any $D>C^23^\alpha$.
\end{lemma}

\begin{proof}
Let $x\in\supp\mu$ and $r>0$. Then, by \eqref{Ahlfors},
$\mu(B(x,3r))\le C(3r)^\alpha$ and $\mu(B(x,r))\ge C^{-1}r^\alpha$. Consequently,
$\mu(B(x,3r))/\mu(B(x,r))\le C^23^\alpha$ and thus $\mu$ is $D$-Federer for any $D>C^23^\alpha$.
\end{proof}

\begin{lemma}\label{lem613B}
Suppose that $\mu$ is a $(C,\alpha)$-Ahlfors regular measure on $\R$.
Then $\mu$ is absolutely $\left(2^\alpha C^2,\alpha\right)$-decaying on $\R$.
\end{lemma}

\begin{proof}
To begin with note that any hyperplane in $\R$ is just a singleton. Take any open ball $B$ of radius $r>0$ centred in $\supp\mu$, any $y\in\R$ and any $\varepsilon>0$. If there exists $y'\in B(y,\varepsilon)\cap\supp\mu$ then $B(y,\varepsilon)\subset B(y',2\varepsilon)$ and, since $\mu(B(y,\varepsilon))\le \mu(B(y',2\varepsilon))$, by \eqref{Ahlfors},
we have that
$
\mu(B(y,\varepsilon))\le C(2\varepsilon)^\alpha\,.
$
Clearly, this also holds if $y'$ does not exist, as in this case $\mu(B(y,\varepsilon))=0$. Again, by \eqref{Ahlfors},
$\mu(B)\ge C^{-1}r^\alpha$. Then,
$$
\mu(B\cap B(y,\varepsilon))\le \mu(B(y,\varepsilon))\le C(2\varepsilon)^\alpha \le C(2\varepsilon)^\alpha \frac{\mu(B)}{C^{-1}r^\alpha}=2^\alpha C^2\left(\frac{\varepsilon}{r}\right)^\alpha\mu(B)\,.
$$
Therefore $\mu$ is absolutely $\left(2^\alpha C^2,\alpha\right)$-decaying.
\end{proof}

\begin{proof}[Proof of Theorem~\ref{thm:QnD-fractal-local}]
Suppose that $\mu$, $\varphi$ and $I_0$ are the same as in {\rm\S\ref{prelim}} and let
\begin{equation}\label{vb10}
\HH_0(x)=z(x) u(\varphi(x))\qquad\text{and}\qquad\HH(x)=g_{\vv \tau}\HH_0(x)\,,
\end{equation}
where $z(x)$ is given by \eqref{eq:z(x)}, $u(\varphi(x))$ is given by \eqref{u()} and $g_{\tau}$ is given by \eqref{gtau}.
Observe that
\begin{equation*}
W(\tau,J,\delta)= \left\{ x \in J: \lat{h(x)}\not\in K_{\delta} \right\}\,.
\end{equation*}
By Lemmas~\ref{lem613} and \ref{lem613B}, $\mu$ is absolutely decaying and $D$-Federer with $D>C^23^\alpha$. Since $\varphi$ is analytic, $h_0$ is analytic. Then, Theorem~\ref{thm:QnD-fractal-local} becomes a special case of Theorem~\ref{KleinbockQnD3}. The proof is thus complete.
\end{proof}

\subsection{Proof of the global estimate}
The goal now is to prove Theorem~\ref{thm:QnD-fractal-global}. We begin with several auxiliary lemmas.

\begin{lemma}[Laplace identity, {\cite[p.\,105]{MR568710}}]\label{Laplace}
For any $1\le r\le n+1$ and any vectors $\vv u_1,\dots,\vv u_r,\vv v_1,\dots,\vv v_{r}\in\R^{n+1}$ we have that
$$
|(\vv u_1\we\cdots\we\vv u_r)\cdot(\vv v_1\we\cdots\we\vv v_r)|=\left|\det\left(\vv u_i\cdot\vv v_j\right)_{1\le i,j\le r}\right|\,.
$$
\end{lemma}

\begin{lemma}\label{bot}
Let $1\le r\le n$, vectors $\vv v_1,\dots,\vv v_r\in\R^{n+1}$ and $\vv u_1,\dots,\vv u_{n+1-r}\in\R^{n+1}$ be linearly independent and satisfy the following two conditions\/{\rm:}
\begin{equation}\label{vb17}
\text{$\vv v_i\cdot\vv u_j=0$\quad $(1\le i\le r$, $1\le j\le n+1-r)$}
\end{equation}
and
\begin{equation}\label{vb18}
\|\vv v_1\we\cdots\we\vv v_{r}\|=\|\vv u_1\we\cdots\we\vv u_{n+1-r}\|\,.
\end{equation}
Then for any $\vv w_1,\dots,\vv w_r\in\R^{n+1}$
\begin{equation}\label{vb22}
|(\vv w_1\we\cdots\we\vv w_r)\cdot(\vv v_1\we\cdots\we\vv v_r)|=\|\vv w_1\we\cdots\we\vv w_r\we\vv u_1\we\cdots\we\vv u_{n+1-r}\|\,.
\end{equation}
\end{lemma}

\begin{proof}
Equation \eqref{vb22} is well known. For example, it is a partial case of Equation (3.10) in \cite{MR2874641}, bearing in mind that conditions \eqref{vb17} and \eqref{vb18} mean that $\vv v_1\we\cdots\we\vv v_r$ is $\pm$ the Hodge dual to $\vv u_1\we\cdots\we\vv u_{n+1-r}$. The latter can be seen using Lemmas~3.1 and 3.2 in \cite{MR2874641} and the well known fact that the Hodge operator is an isometry, see \cite[p.201]{MR2874641}.
\end{proof}

\begin{lemma}[Lemma~5G in \cite{MR1176315}]\label{hodge}
Let $1\le r\le n$ and $\vv v_1,\dots,\vv v_r\in\Z^{n+1}$ be a primitive collection. Then there exists a primitive collection $\vv u_1,\dots,\vv u_{n+1-r}\in\Z^{n+1}$ satisfying \eqref{vb17} and \eqref{vb18}.
\end{lemma}

\begin{prop}\label{mainprop}
Let $\varphi$ and $I_0$ be the same as in {\rm\S\ref{prelim}}, $\HH$ be given by \eqref{vb10} and $r\ge2$. Then for any primitive collection $\vv v_1,\dots,\vv v_r\in\Z^{n+1}$ there exists a multi-index $I=\{1,2,i_3,\dots,i_r\}\subset\{1,\dots,n+1\}$ and $\vv a\we\vv b\in\bigwedge^2\left(\Z^{n+1}\right)\setminus\{\vv0\}$ such that
$$
|(\HH(x)\vv v_1\wedge\cdots\wedge \HH(x)\vv v_r)_I|=e^{\sum_{i\in I}\tau_i}\left|\left(\tilde\varphi(x)\wedge\tilde\varphi'(x)\right)\cdot(\vv a\wedge\vv b)\right|\,,
$$
where
$\tilde\varphi(x)=(1,\varphi(x))=(1,x,\varphi_2(x),\dots,\varphi_n(x))$.
\end{prop}

\begin{proof}
Let $\vv v_1,\dots,\vv v_r\in\Z^{n+1}$ be a primitive collection, $r\ge2$. By Lemma~\ref{hodge}, choose a primitive collection $\vv u_1,\dots,\vv u_{n+1-r}\in\Z^{n+1}$ satisfying \eqref{vb17} and \eqref{vb18}. If $r=n+1$ this will be empty. For $i=1,\dots,n+1$ let $\vv h_i$ denote the $i$th row of $h$. Using the explicit form of $z(x)$ and $u(\varphi(x))$ given by \eqref{phi}, \eqref{u()} and \eqref{eq:z(x)}, we readily calculate that $\vv h_1=e^{\tau_1}\tilde\varphi(x)$, $\vv h_2=e^{\tau_2}\tilde\varphi'(x)$ and $\vv h_i=e^{\tau_i}\vv e_i$ for $i=3,\dots,n+1$.
Therefore, for every $I=\{1,2,i_3,\dots,i_r\}\subset\{1,\dots,n+1\}$ we have that
\begin{align}
\nonumber |\HH(x)&\vv v_1\wedge\cdots\wedge \HH(x)\vv v_r)_I|
=\left\lvert\det\left(\vv h_{i_\ell}\cdot\vv v_j\right)_{1\le \ell,j\le r}\right\rvert\\[1ex]
\nonumber &\stackrel{\text{Lemma~\ref{Laplace}}}{=}e^{\sum_{i\in I}\tau_i}\left\lvert\left(\tilde\varphi(x)\we\tilde\varphi'(x)\wedge\vv e_{i_3}\cdots\wedge \vv e_{i_r}\right)\cdot(\vv v_1\we\cdots\we\vv v_{r})\right\rvert\\[1ex]
&\stackrel{\text{Lemma~\ref{bot}}}{=}e^{\sum_{i\in I}\tau_i}\left\|\tilde\varphi(x)\we\tilde\varphi'(x)\wedge\vv e_{i_3}\cdots\wedge \vv e_{i_r}\wedge\vv u_1\we\cdots\we\vv u_{n+1-r}\right\|\,.\label{wedge1}
\end{align}
Choose $i_3,\dots,i_r\ge3$ so that
\begin{equation}\label{vb23}
\vv e_{i_3}\we\cdots\we\vv e_{i_r}\we\vv u_1\we\cdots\we\vv u_{n+1-r}\in\textstyle\bigwedge^{n-1}\left(\Z^{n+1}\right)\neq\vv0\,.
\end{equation}
By Lemma~\ref{hodge}, there exist vectors $\vv a,\vv b\in\Z^{n+1}$ orthogonal to every vector appearing in \eqref{vb23} such that $\|\vv a\we\vv b\|$ is equal to the norm of the multi-vector in \eqref{vb23}.
Then, using Lemma~\ref{bot}, we find that
$$
\left\|\tilde\varphi(x)\we\tilde\varphi'(x)\wedge\vv e_{i_3}\cdots\wedge \vv e_{i_r}\wedge\vv u_1\we\cdots\we\vv u_{n+1-r}\right\|=|(\tilde\varphi(x)\we\tilde\varphi'(x))\cdot(\vv a\we\vv b)|\,.
$$
This together with \eqref{wedge1} completes the proof.
\end{proof}

\begin{prop}\label{prop66}
Let $\varphi$ and $I_0$ be the same as in {\rm\S\ref{prelim}}, $\HH$ be given by \eqref{vb10}. Further suppose that \eqref{tau} and \eqref{tau_cond} hold. Then there exists $0<\rho_*\le1$ such that for any primitive collection $\vv v_1,\dots,\vv v_r\in\Z^{n+1}$
$$
\sup_{x\in I_0}\left\|\HH(x)\vv v_1\wedge\cdots\wedge \HH(x)\vv v_r\right\|\ge\rho_*.
$$
\end{prop}

\begin{proof}
First let $r=1$. Take any non-zero $\vv v_1\in\Z^{n+1}$. By the non-degeneracy of $\varphi$, we have that
$\tilde\varphi(x)\cdot\vv v_1$  is not identically zero. Then, using $\tau_1>0$ and $\|\vv v_1\|\ge1$ we have that
\begin{align}
\sup_{x\in I_0}\|\HH(x)\vv v_1\|&\ge e^{\tau_1}\sup_{x\in I_0}|\tilde\varphi(x)\cdot\vv v_1|{\ge}
\inf_{\|\vv v\|=1}\sup_{x\in I_0}|\tilde\varphi(x)\cdot\vv v|=:\rho_1\,.\label{rho1}
\end{align}
Observe that $\rho_1>0$, since in \eqref{rho1} we are taking infimum over a compact set of a positive continuous function of $\vv v$.

Now let $r\ge2$ and $\vv v_1,\dots,\vv v_r\in\Z^{n+1}$ be any primitive collection. Let $I=\{1,2,i_3,\dots,i_r\}$ and $\vv a\we\vv b\in\bigwedge^2(\Z^{n+1})\setminus\{\vv0\}$ arise from Proposition~\ref{mainprop}.
By Lemma~\ref{Laplace} and the fact that $\vv w_1(x)=\tilde\varphi(x)$, $\vv w_2(x)=\tilde\varphi'(x)$, we have that
\begin{equation*}
\left(\tilde\varphi(x)\wedge\tilde\varphi'(x)\right)\cdot(\vv a\wedge\vv b)=
\left|\begin{array}{cc}
\tilde\varphi(x)\cdot\vv a & \tilde\varphi(x)\cdot\vv b \\
\tilde\varphi'(x)\cdot\vv a & \tilde\varphi'(x)\cdot\vv b
\end{array}
\right|\,.
\end{equation*}
This is a Wronskian of two linearly independent analytic functions and hence it must be non-zero.
Therefore,
\begin{align}
\nonumber\sup_{x\in I_0}\|\HH(x)&\vv v_1\wedge\cdots\wedge \HH(x)\vv v_r\|\ge e^{\sum_{i\in I}\tau_i}\sup_{x\in I_0}\left|\left(\tilde\varphi(x)\wedge\tilde\varphi'(x)\right)\cdot(\vv a\wedge\vv b)\right|\\
&\stackrel{\eqref{tau}\&\,\eqref{tau_cond}}{\ge} \sup_{x\in I_0}\left|\left(\tilde\varphi(x)\wedge\tilde\varphi'(x)\right)\cdot(\vv a\wedge\vv b)\right|\nonumber\\
&\stackrel{\|\vv a\we\vv b\|\ge1}{\ge}\inf_{\|\vv u\we\vv v\|=1}\sup_{x\in I_0}\left|\left(\tilde\varphi(x)\we\tilde\varphi'(x)\right)\cdot(\vv u\we \vv v)\right|=:\rho_2\,.\label{rho2}
\end{align}
Observe that $\rho_2>0$, since in \eqref{rho2} we are taking infimum over a compact set of a positive continuous function of $\vv u\we\vv v$. Taking $\rho_*=\min\{1,\rho_1,\rho_2\}$ and putting together \eqref{rho1} and \eqref{rho2} completes the proof of the proposition.
\end{proof}

\begin{proof}[Proof of Theorem~\ref{thm:QnD-fractal-global}]
In view of Proposition~\ref{prop66}, Theorem~\ref{thm:QnD-fractal-global} follows immediately from Theorem~\ref{thm:QnD-fractal-local} with $\gamma$ as in Theorem \ref{thm:QnD-fractal-local} and $M_2=M_1\rho_*^{-\gamma}\mu(3I_0)$.
\end{proof}

\section{Final remarks}

While the primary purpose of this paper is to resolve specific problems in the theory of Diophantine approximation, the methods we presented lay the foundation for a comprehensive theory of badly approximable points and of bounded orbits of diagonal flows on homogenous spaces.
The next natural step in the direction of such a theory is to understand whether the sets $\Bad(\rr)$ \,\,(this time unrestricted to any submanifold of $\R^n$) are winning. As we mentioned in the introduction it is currently known from \cite{MR3910474} that $\Bad(\rr)$ is hyperplane absolute winning in the case $r_1=\dots=r_{n-1}\ge r_n$. Further developing the ideas of this paper in \cite{BNY2} we establish the following unconditional result.

\begin{thm}
For any $n$-tuple of weights of approximation $\rr$ the set $\Bad(\rr)$ is hyperplane absolute winning.
\end{thm}

In particular, in \cite{BNY2} we further develop the framework of intersections with fractals which requires suitable extensions of Lemmas~\ref{thm:intersection-fractals-winning} and \ref{diffuse} underlining the equivalence \eqref{equiv}. Also, in \cite{BNY2} we demonstrate an equivalent approach that uses Cantor winning sets instead of generalised Cantor sets discussed in \S\ref{sec-cantor-like-construction} above.

We end this paper by noting several natural follow-up problems.

\begin{prob}\label{p2}
Generalise the results of this paper to arbitrary non-degenerate curves that are not analytic.
\end{prob}

Most of the proof presented this paper will work for the non-analytic case. The sticking `technical' point is to obtain Corollary~\ref{prop614+} in the non-analytic case, which would require generalising Proposition~\ref{prop614} to the skew-gradients introduced in \cite{Bernik-Kleinbock-Margulis}.

\begin{prob}[{\cite[Conjecture~D]{MR3231023}}]\label{p3}
Prove that the set of $\rr$-badly approximable points lying on any nondegenerate submanifold of $\R^n$ is winning.
\end{prob}

This more general version of Problem~\ref{p1}, which we will attempt to address in subsequent publications, will require further generalisation of the framework of intersection with fractals, adapting the notion of hyperplane absolute winning and developing suitable quantitative non-divergence estimates.


Finally we note that all of the aforementioned problems can be extended to other settings: inhomogeneous approximations, Diophantine approximation over locally compact field, $p$-adic and more generally $S$-arithmetic setting. Examples of Diophantine approximation in these settings can be found for instance in \cite{MR2321374, MR3733884, MR2314053}. To what extent the techniques presented in this paper can be generalised to these settings remains an appealing open question.


\bibliography{reference}
\bibliographystyle{alpha}
\end{document}